\documentclass[12pt]{amsart}
\usepackage{amsthm, amstext, amsmath, amssymb, graphicx, verbatim, fullpage,mathtools,booktabs,hyperref,mathrsfs,algorithm,xcolor,versions,enumerate,enumitem}

\usepackage{tikz-cd}
\usepackage{colonequals}

\usepackage[utf8]{inputenc}
\usepackage[english]{babel}

\newtheorem{thm}{Theorem}[section]
\newtheorem{lemma}[thm]{Lemma} \newtheorem{cor}[thm]{Corollary}
\newtheorem{prop}[thm]{Proposition}
\theoremstyle{definition}
\newtheorem{defn}[thm]{Definition}

\newtheorem{conj}[thm]{Conjecture}

\newtheorem{remark}[thm]{Remark}

\newcommand{\A}{\mathbb{A}}

\newcommand{\N}{\mathbb N}
\newcommand{\R}{\mathbb R}
\newcommand{\C}{\mathbb C}
\newcommand{\Q}{\mathbb Q}

\newcommand{\D}{\mathbb D}

\newcommand{\PP}{\mathbb{P}}

\newcommand{\ovl}{\overline}

\DeclareMathOperator{\charac}{char}

\DeclareMathOperator{\Spec}{\text{Spec}}
\DeclareMathOperator{\PGL}{PGL}

\DeclareMathOperator{\Aut}{Aut}

\newcommand{\abs}[1]{\left|#1\right|}

\newcommand{\mc}{\mathcal}

\newcommand{\rev}[1]{#1}

\DeclareMathOperator{\Top}{top}

\DeclareMathOperator{\Ind}{Ind}
\DeclareMathOperator{\Exc}{Exc}
\DeclareMathOperator{\adj}{adj}

\DeclareMathOperator{\Div}{Div}

\DeclareMathOperator{\Num}{Num}

\DeclareMathOperator{\rad}{rad}

\DeclareMathOperator{\Glu}{Glutsyuk}

\DeclareMathOperator{\gen}{gen}

\DeclareMathOperator{\str}{str}

\newcommand{\bk}{{\mathbf{k}}}

\tikzset{>=Straight Barb, commutative diagrams/arrow style=tikz}

\usetikzlibrary{arrows}
\usetikzlibrary{calc}

\DeclareMathOperator{\Mid}{mid}

\usepackage[pdf]{pstricks}
\usepackage{pstricks-add}
\usepackage{pst-func}
\usepackage{pst-plot}

\email{maxhweinreich@gmail.com}
\address{Department of Mathematics, Harvard University, Cambridge, MA 02138.
 ORCID: 0000-0002-0103-2245}

\begin{document}
 
\title{Algebraic Billiards in the Fermat Hyperbola}
\author{Max Weinreich}

\begin{abstract}
We prove two results on the algebraic dynamics of billiards in generic algebraic curves of degree $d \geq 2$. First, the dynamical degree grows quadratically in $d$; second, the set of complex periodic points has measure 0, implying the Ivrii Conjecture for the classical billiard map in generic algebraic domains. To prove these results, we specialize to a new billiard table, the Fermat hyperbola, on which the indeterminacy points satisfy an exceptionality property. Over $\C$, we construct an algebraically stable model for this billiard via an iterated blowup. Over more general fields, we prove essential stability, i.e. algebraic stability for a particular big and nef divisor.
\end{abstract}

\maketitle

\section{Introduction} \label{sect_intro}

The classical billiard map is a discrete-time dynamical system that models a point particle bouncing around inside a  \emph{billiard table}, a plane region $\Omega \subset \R^2$ with smooth boundary \cite{MR2168892}. The domain of the billiard map is a subset of the real surface $\partial \Omega \times S^1$, the set of unit-length vectors based on $\partial \Omega$. The first factor tracks the position of the ball upon first collision with the boundary $\partial \Omega$, and the second factor tracks the ball's direction of movement. After each collision, the direction of the ball is reflected across the tangent line to $\Omega$ at the collision point.

A major goal of billiards research is to show that a generic billiard is chaotic, where the precise meaning of chaos depends on the problem of interest \cite{MR3388585}. In this article, we prove two theorems of this type for algebraic billiards.

Algebraic billiards is an extension of real billiards to algebraically closed fields $\bk$ \cite{MR3236494, billiardsI}. The role of the billiard table is played by a given curve $C \subset \PP^2_\bk$ of degree $d \geq 2$, here assumed smooth for simplicity. The role of the space of directions $S^1$ is played by the compactified unit tangent space $D$ of a chosen nondegenerate quadratic form $\Theta$. Thus $D$ is a plane conic, abstractly isomorphic to $\PP^1$. The domain $C \times D$ is an algebraic surface, and is a compactification of the variety of unit tangent vectors for $\Theta$ based on $C \cap \A^2_\bk$.

The billiards correspondence $b = b_{C,D}$ is a rational correspondence, i.e. a multivalued rational map, denoted
$$ b_{C,D} : C \times D \vdash C \times D.$$
\rev{
We define it in two steps, modeling the operations of moving the ball and reflecting the direction. Consider a general input $(x, v) \in C \times D$. If $x$ is in $\A^2_{\bk}$, the \emph{line $\ell(x,v)$ through $x$ of direction $v$} is the line in $\PP^2_\bk$ through $x$ with slope determined by $v$. If $v$ meets $C$ transversely at $x$, the set
$$s_{C,D} (x, v) \colonequals \{(x',v): x' \in C \cap \ell(x,v) \smallsetminus x \}$$
has $d - 1$ elements, counted with multiplicity. The operation $(x,v) \mapsto s_{C,D}(x,v)$ is therefore a multivalued map defined generically on $C \times D$. Taking the Zariski closure of its graph, we obtain a rational correspondence on $C \times D$, the \emph{secant correspondence} 
$$s_{C, D}: C \times D \vdash C \times D.$$
The \emph{reflection map} $r = r_{C,D}$ is a birational involution
$$r_{C,D} : C \times D \dashrightarrow C \times D$$
of the form $r_{C,D}(x,v) = (x, r_x(v))$, i.e. a rational family of involutions $r_x: D \to D$ over $C$. The map $r_x$ is reflection across the tangent line to $C$ at $x$ with respect to the chosen quadratic form $\Theta$; see Section \ref{sect_setup}. 
 
The billiards correspondence is the composite of secant and reflection,
$$b_{C,D} \colonequals r_{C,D} \circ s_{C,D}.$$
It is a rational $(d-1)$-to-$(d-1)$ correspondence, meaning that a general point of $C \times D$ has $d-1$ images and $d-1$ preimages. See Figure \ref{fig_ab}.
}
\begin{figure}
    \centering
    \includegraphics[height=2in]{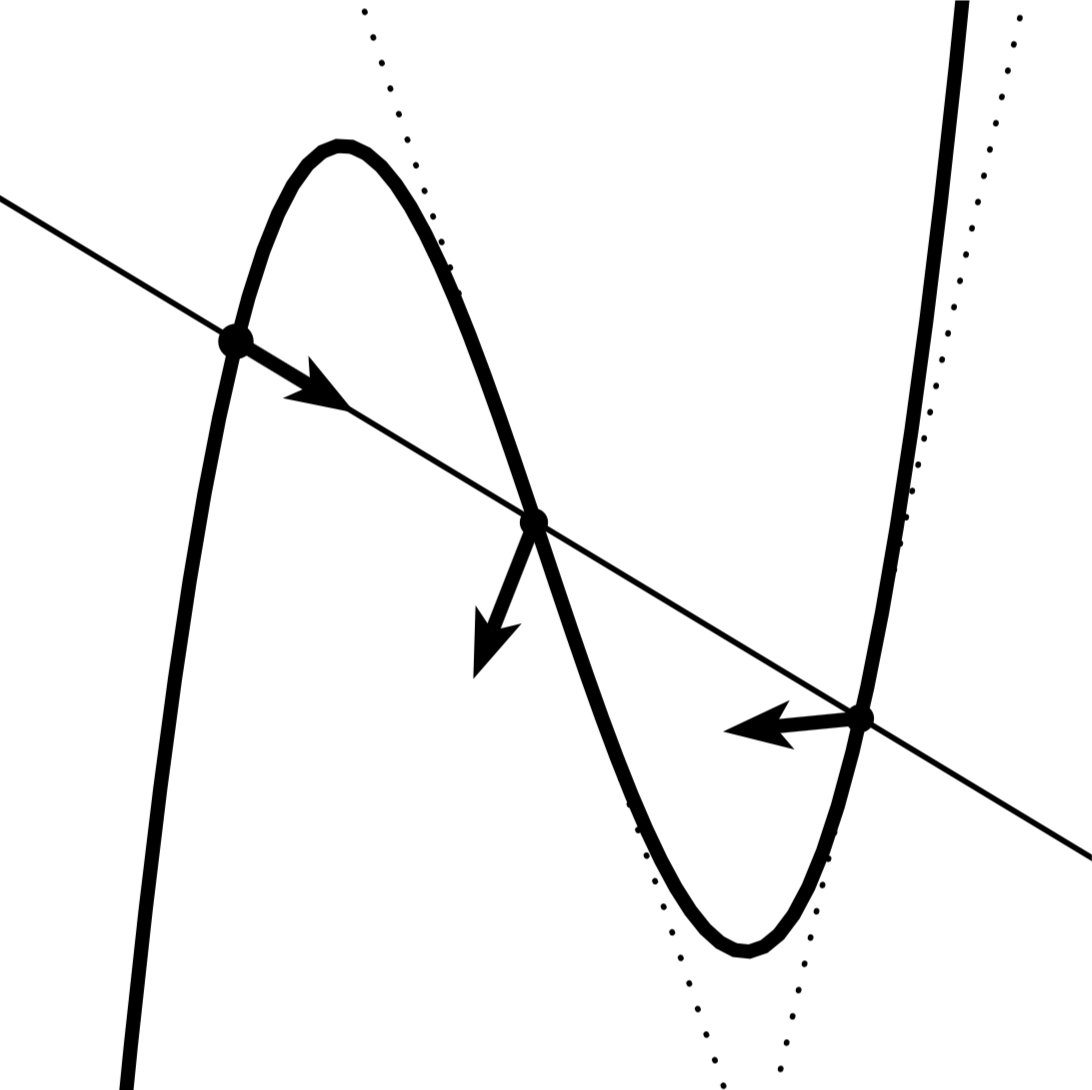}
    \caption{Algebraic billiards in the curve $C$ defined by $x_1 = x_0^3$ \rev{relative to the quadratic form} $\Theta = q_0^2 + q_1^2$. A generic input (represented by the leftmost vector) has two images (represented by the middle and rightmost vectors).}
    \label{fig_ab}
\end{figure}

In the algebraic setting, one often may prove properties of the generic object in a family by a judicious choice of specialization. However, describing the dynamics of any particular algebraic billiard is a difficult problem due to the multivaluedness of the correspondence. In this article, we conduct a detailed study of the dynamics of the \emph{Fermat hyperbola billiard}, defined by the pair
$$C: \; (x_0 - i x_1)^d + (x_0 + i x_1)^d = 1,$$
$$D: \; q_0^2 + q_1^2 = 1.$$
We are ultimately less interested in the Fermat hyperbola itself than in the information about generic billiards that we can derive from it. We give two applications of our analysis: dynamical degrees and periodic orbits.

\subsection{Dynamical degrees}

Topological and metric entropy of classical billiards is a well-studied and delicate subject \cite{MR3754521, MR896765, MR4213303, cinc2022, MR872698, MR2609011}. The dynamical degree is an algebraic analogue of entropy that has close connections to topological entropy, Arnold entropy, and arithmetic entropy. \rev{Dynamical degrees were introduced for complex correspondences by Dinh-Sibony \cite{MR2391122}. In this paper, we use a version due to Truong that measures divisor growth for correspondences over arbitrary algebraically closed fields \cite[Theorem 1.1 (1)]{MR4048444}.} In billiards, the dynamical degree controls degree growth in classical sequences of curves associated to billiard systems, such as wavefronts and caustics of reflection \cite{billiardsI}.

\begin{defn}[\cite{MR4048444}]
    Let $X$ be a smooth projective surface equipped with a dominant rational correspondence $f: X \vdash X$. Each iterate $f^n$ induces a pushforward homomorphism $(f^n)_* : \Num X \to \Num X$ on the group of divisor classes on $X$ modulo numerical equivalence \rev{(see Section  \ref{sect_act_div} for details).} Choosing any ample class $\Delta$ in $\Num X$, the (first) \emph{dynamical degree} of $f$ is the limit
$$\lambda_1(f) \colonequals \lim_{n \to \infty} ((f^n)_* \Delta \cdot \Delta)^{1/n}.$$
This limit exists, is independent of $\Delta$, and is a birational conjugacy invariant.
\end{defn}
More generally, a dominant rational correspondence $f$ of a smooth projective variety of dimension $N$ has a sequence of dynamical degrees $\lambda_i(f)$, where $i = 0, \ldots, N$.
The dynamical degrees $\lambda_0(f)$ and $\lambda_N(f)$ are given by the bidegree of the correspondence, so in the setting of billiards, we have $\lambda_0(b) = \lambda_2(b) = d - 1$, leaving only $\lambda_1(b)$ to compute.

In this article, we calculate the dynamical degree of the billiard in the Fermat hyperbola. This gives a lower bound on the dynamical degree of the billiards correspondence in the generic algebraic curve of degree $d$.

\begin{thm} \label{thm_main_dd}
    Fix $d \geq 2$, and let $\bk$ be an algebraically closed field such that \rev{$\charac \bk = 0$ or} $\gcd(\charac \bk, 2d) = 1$. Let $C_{\gen}$ be the generic plane curve of degree $d$, and let $D$ be the unit tangent space of a nondegenerate quadratic form $\Theta$.
    The dynamical degree of the billiards correspondence $b_{\gen} : C_{\gen} \times D \vdash C_{\gen} \times D$ satisfies
    $$ \lambda_1(b_{\gen}) \geq \frac{2d^2 - 3d + \sqrt{(2d^2 - 3d)^2 - 4(d-1)}}{2} \geq 2d^2 - 3d - 1.$$
\end{thm}

To contextualize our work, we recall the Birkhoff Conjecture, a long-standing open problem in billiards. Chaos in a dynamical system is measured by topological entropy $h_{\Top} \geq 0$. In classical billiards, the Birkhoff Conjecture predicts that the ellipse is the unique strictly convex, smooth curve for which the billiard has $h_{\Top} = 0$; see \cite{MR3388585}. Yet it is difficult to show that $h_{\Top} > 0$ for any particular non-elliptical table. On the other hand, it was recently shown that $h_{\Top} > 0$ for a $\mc{C}^2$-generic set of convex billiards \cite{MR4710877}. Theorem \ref{thm_main_dd} is an algebraic analogue of that result. A different analogue, the Polynomial Birkhoff Conjecture for complex billiards, concerning the existence of polynomial first integrals, was recently proved by Glutsyuk \cite{MR4210728}.

Calculating the exact value of $\lambda_1(b_{\gen})$ is a difficult problem due to the existence of destabilizing orbits, a frequent issue in algebraic dynamics. In \cite{billiardsI}, we identify and resolve some destabilizing orbits of $b$ for sufficiently general curves $C$, providing an upper bound on $\lambda_1(b)$ for those curves. There is no reason to expect other destabilizing orbits, suggesting the following conjecture.
\begin{conj}[\cite{billiardsI}] \label{conj_main}
    The dynamical degree of the generic billiard is $\lambda_1(b_{\gen}) = \rho_d$, where  $\rho_d$ is the largest absolute value of the roots of the polynomial
$$\lambda^3 - (2 d^2 - d - 3) \lambda^2 + (2d^2 - 4d + 3) \lambda - (d - 1).$$
\end{conj}
The main result of \cite{billiardsI} shows that $\lambda_1(b_{\gen}) \leq \rho_d < 2d^2 - d - 3$, so the dynamical degree is quadratic in $d$.
Ruling out further destabilizing orbits seems to be a difficult problem due to exponential growth in correspondence orbits over time.

\subsection{Periodic orbits}

The Ivrii Conjecture is among the central open problems in the theory of billiards \cite{MR3388585}. It is motivated by the long-standing Weyl conjecture on the second-order asymptotics of the spectrum of the Laplacian in a bounded planar domain $\Omega \subset \R^2$ with Dirichlet boundary conditions and smooth boundary, proposed in \cite{Weyl1911}. Weyl conjectured that the number $N(L)$ of Laplacian  eigenvalues of absolute value at most $L$ is
$$N(L) = \frac{\textrm{Area } \Omega}{4\pi } L - \frac{\textrm{Length }{\partial \Omega}}{4 \pi} \sqrt{L} + o(\sqrt{L}).$$
See \cite[Section 3.3]{MR4655924} for an introduction to the topic. Ivrii proved that if the set of periodic points for the billiard in $\Omega$ has measure $0$, then the Weyl conjecture holds for $\Omega$ \cite{MR575202}. This innocuous-sounding property concerning periodic orbits has turned out to be surprisingly resistant to proof, and thus one conjecture has produced another.

\begin{conj}[Ivrii Conjecture] \label{conj_i}
    The set of periodic points of a billiard in a plane region in $\R^2$ with smooth boundary has measure $0$.
\end{conj}

 The main settled cases of the conjecture are as follows:
\begin{enumerate}
    \item strictly convex domains $\Omega$ with globally regular analytic boundary \cite{MR775930},
    \item concave piecewise-analytic domains \cite{MR829600},
    \item generic domains $\Omega$ with smooth boundary, where ``generic'' means that the claim holds on a residual subset in the Whitney $\mc{C}^\infty$-topology \cite{MR939062}.
\end{enumerate}
The sets of $3$-periodic and $4$-periodic points are known to have measure $0$ \cite{MR1001275, MR2988811}. Period $5$ remains open, an indication of the difficulty of the problem; see \cite{MR2988811} for some progress in this direction and a detailed survey.

In the setting of complex billiards on irreducible tables, the Ivrii conjecture becomes more rigid, because algebraic conditions that hold on sets of positive measure hold everywhere. An algebraic billiard is \emph{$n$-reflective} if every point in the domain can be continued to an $n$-periodic orbit along some branch of the correspondence. Glutsyuk proposed the problem of classifying $n$-reflective tables as a complexification of the Ivrii conjecture, and solved several cases: $3$-periodic and $4$-periodic orbits, and odd-periodic orbits on curves with no isotropic points at infinity \cite{MR3224419,MR3236494}.

Our second main result proves the Ivrii conjecture for the generic complex billiard of degree $d \geq 2$, for all periods $n$, by specialization to the Fermat hyperbola. Consequently, Conjecture \ref{conj_i} is true for very general algebraic tables over $\R$, including non-convex tables.

\begin{thm} \label{thm_main_ivrii}
\leavevmode
\begin{enumerate}
        \item
    Let $\bk = \C$ and $d \geq 2$. For all $n \in \N$, the complex billiard 
    $$b_{\gen} : C_{\gen} \times D \vdash C_{\gen} \times D$$ in the generic algebraic curve of degree $d$ is not $n$-reflective. \label{it_ivrii_cplx}
    \item
    Let $T \subset \R^2$ be a real algebraic plane curve defined by the vanishing of a polynomial $F(x,y)$ of degree $d \geq 2$ with algebraically independent coefficients over $\Q$. Let $\Omega$ be a bounded component of $\R^2 \smallsetminus T$. Then the set of periodic points of the classical billiard map inside $\Omega$ has measure $0$. \label{it_ivrii}
\end{enumerate}
\end{thm}

It follows that the Weyl conjecture holds for these real domains. 

\subsection{Sketch of the proof}
In proving properties of the generic billiard, we have freedom to choose a good specialization.
The Fermat hyperbola billiard has not been studied before, yet it is a natural specialization to consider due to the structure of its indeterminacy points\rev{, described in Section \ref{sect_setup} and Section \ref{sect_fh}.}

Recall the definition $b = r \circ s$. \rev{The sets of indeterminacy points $\Ind s, \Ind r$ are related to basic geometric properties of $C$ and $D$; see Lemma \ref{lem_s_basic} and Lemma \ref{lem_r_basic}.} For general billiards, one can show that
$$\Ind b = \Ind s \cup \Ind r.$$
Further, for general billiards, the correspondences $s$ and $r$ each fix the indeterminacy points of the other:
$$p \in \Ind s \implies r(p) = p,$$
$$p \in \Ind r \implies s(p) \ni p.$$ 
The orbits ending in $\Ind s$ are easy to describe, but the orbits ending in $p \in \Ind r$ are hard to control because $s^{-1}(p)$ in general contains points besides $p$. The key property of the Fermat hyperbola is that, due to the structure of its \rev{tangent lines of maximal order,} all points in $\Ind r_{C,D}$ are $s_{C,D}$-exceptional (i.e. $s_{C,D}^{-1}(p) = \{p\}$); see Lemma \ref{lem_e}. This vastly simplifies the indeterminate orbits. The indeterminacy structure of the Fermat hyperbola billiard allows us to construct good birational models for the dynamics, as we now explain.

\begin{figure}[h]
    \centering
\includegraphics[height=1in]{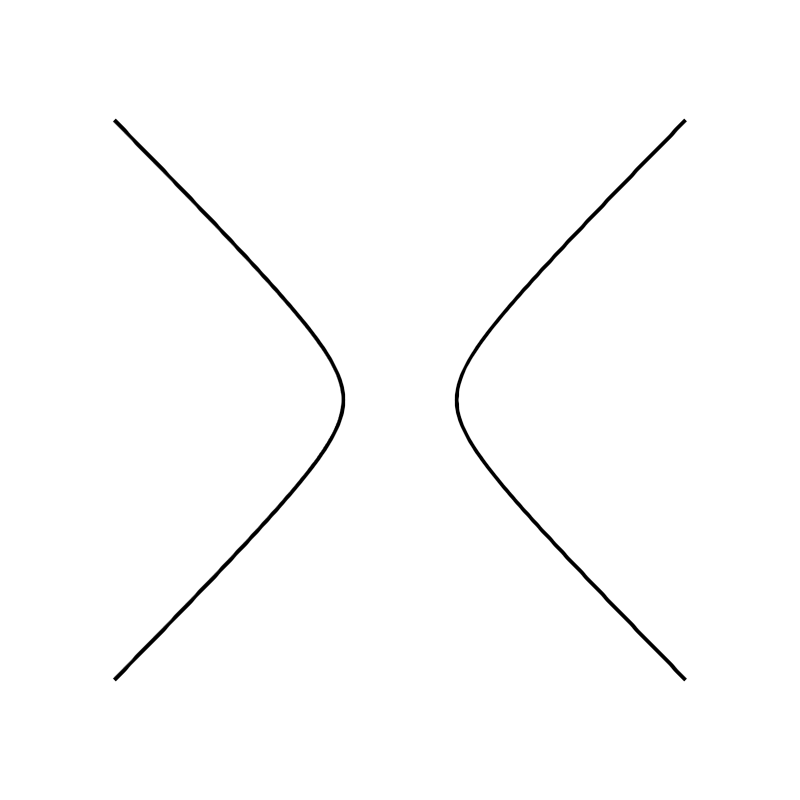}
\hspace{0.5in}
\includegraphics[height=1in]{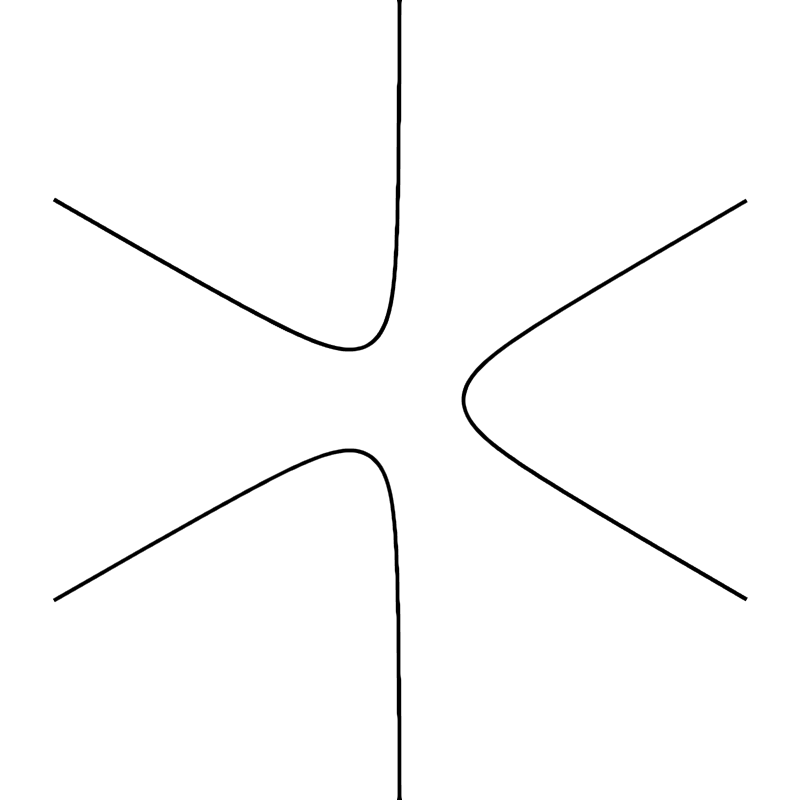}
\hspace{0.5in}
\includegraphics[height=1in]{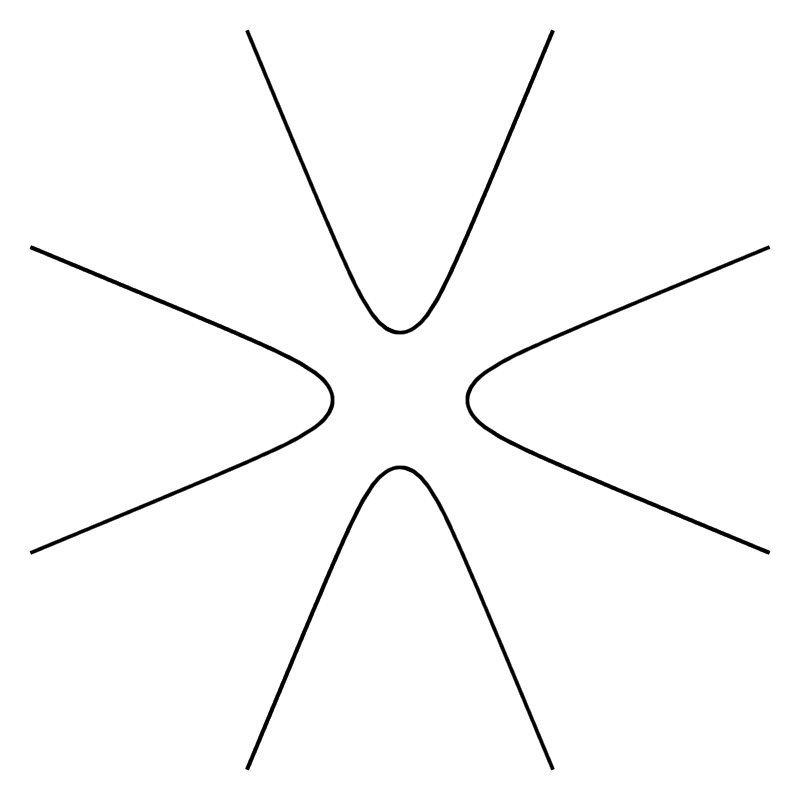}
\hspace{0.5in}
\includegraphics[height=1in]{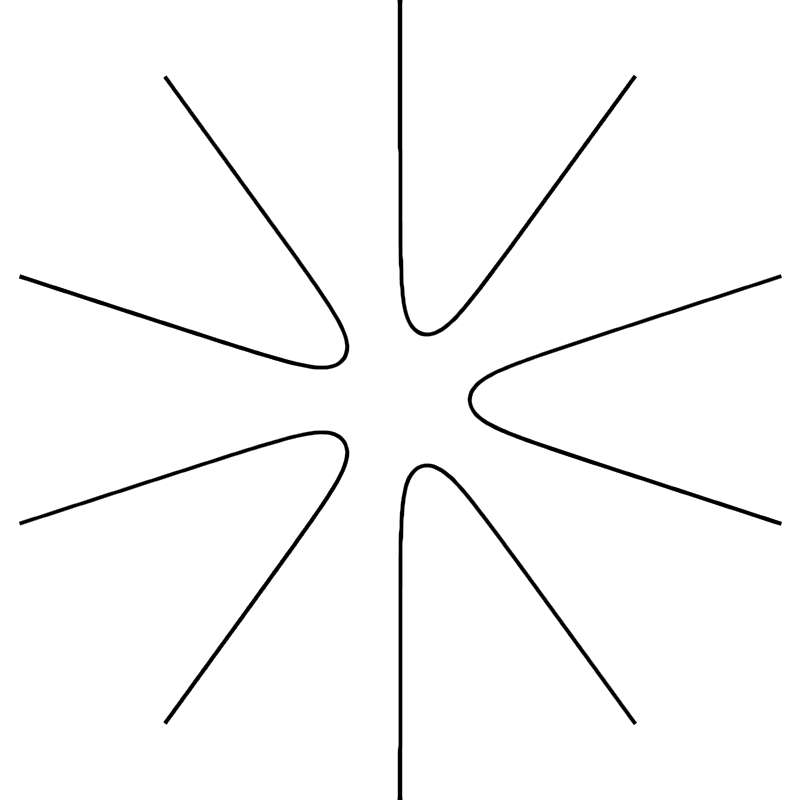}
    \caption{The real locus of the Fermat hyperbola of degree $d=2,3,4,5$.}
    \label{fig_fh}
\end{figure}

A correspondence 
$f \colon X \vdash X$ is \emph{algebraically stable} if, for all $n \in \N$, the pushforward homomorphism $(f^n)_* : \Num X \to \Num X$ satisfies
$(f^n)_* = (f_*)^n$. In the presence of algebraic stability, calculating the dynamical degree is a matter of intersection theory and linear algebra. 

However, the Fermat hyperbola billiard is not algebraically stable. One instead seeks a birational model of $C \times D$ on which the conjugate of $b_{C, D}$ is algebraically stable. We achieve this over $\C$. Over other fields, we find a model with the following weaker yet sufficient property, also noticed in \cite{MR2892921}.

\begin{defn} \label{def_es}
    A surface correspondence $f: X \vdash X$ is \emph{essentially stable} if there exists a big and nef divisor $\Delta$ on $X$ such that, for all $n$,
\[(f^n)_*\Delta \cdot \Delta = (f_*)^n \Delta \cdot \Delta.\]
\end{defn}

Our results on birational models are collected in the following theorem.
\begin{thm} \label{thm_main_model}
     Fix $d \geq 2$, and let $\bk$ be an algebraically closed field such that \rev{$\charac \bk = 0$ or} $\gcd(\charac \bk, 2d) = 1$. \begin{enumerate}
         \item The billiards correspondence $b_{C,D}$ in the Fermat hyperbola of degree $d$ admits an essentially stable model. \label{it_thm_es}
         \item If $\bk = \C$, then $b_{C,D}$ admits an algebraically stable model.  \label{it_thm_as}
         \item If $d = 2$, then $b_{C,D}$ is a completely integrable birational map, conjugate to a translation map on an elliptic surface.  \label{it_thm_d_2}
         \item If $d = 3$, then $b_{C,D}$ admits a regular model. \label{it_thm_d_3}
     \end{enumerate} 
\end{thm}

Item (1) allows us to compute $\lambda_1$, proving Theorem \ref{thm_main_dd}. We consider positive characteristic to emphasize the existence of a purely algebraic proof of Theorem \ref{thm_main_dd}. Item (2) exhibits a smooth algebraic billiard in degree $d > 2$ with an algebraically stable model; one ultimately hopes that such a construction is possible for the generic curve. 
Item (3) is mentioned for completeness; in fact, it holds for general conics, not just the Fermat hyperbola \cite[Corollary 7.11]{billiardsI}. Item (4) is a corollary of the proof of (1).

To construct an algebraically stable model of a correspondence $f$, the classical approach is to repeatedly blow up destabilizing orbits, that is, orbits that start with contracted curves and end at points in the indeterminacy locus $\Ind f$. This process has no guarantee of succeeding, since blowups may introduce new destabilizing orbits.

Instead, we use a non-standard strategy to build a good birational model $P$ for $C \times D$. First, we blow up $\Ind s$. This regularizes $s$. Second, we perform a $(d-1)$-fold iterated blowup, chosen to regularize $r$. \rev{We denote the resulting model by $P$ and let $\pi : P \to C \times D$ be the composition of blowups defining $P$; see Figure \ref{fig_P}. We then study the following correspondences $P \vdash P$:
$$ \hat{s} \colonequals \pi^{-1}\circ s\circ\pi, \quad  \hat{r} \colonequals \pi^{-1}\circ r\circ\pi, \quad \hat{b} \colonequals \pi^{-1}\circ b\circ\pi.$$
We prove that $\hat{b} : P \vdash P$ is essentially stable in Proposition \ref{prop_essentially}. This step involves delicate work in coordinates, since $\hat{s}$ and $\hat{b}$ have indeterminacy points if $d > 3$. Over $\C$, if $d$ is odd, then $\hat{b}$ is algebraically stable; if $d$ is even, then a further modification $\hat{b}_+ : P_+ \vdash P_+$ is algebraically stable.} To prove this, we study the complex dynamics of the correspondence on a $1$-dimensional invariant subset called the midpoint divisor.

The proof of Theorem \ref{thm_main_ivrii} is a straightforward consequence of the same analysis. We show that the midpoint divisor contains \rev{a non-periodic orbit.} Then we apply the standard technique in algebraic dynamics of passing to the generic curve, then to curves defined by algebraically independent coefficients.

\subsection{Related questions}

Glutsyuk studied a complex billiards correspondence $b_{\Glu}$ in which the space of directions is the projectivized tangent space ($\simeq \PP^1)$ rather than the unit tangent space $D$ \cite{MR3224419}. These systems are semiconjugate, but $b_{C,D}$ has the additional feature that it preserves a rational $2$-form \cite{billiardsI}. By semiconjugacy, Theorem \ref{thm_main_dd}, Theorem \ref{thm_main_ivrii}(\ref{it_ivrii_cplx}), and Theorem \ref{thm_main_model} also hold for $b_{\Glu}$.

On the algebraic dynamics side, our work follows a substantial literature dedicated to dynamical degrees and the closely related problem of constructing regular or algebraically stable models for rational maps of surfaces \cite{MR1867314, MR2753603, birkett2022stabilisation}. Dynamical degrees are very difficult to compute, and much more is known about rational maps than about correspondences. Most importantly, there are correspondences for which the dynamical degree sequence fails to be log concave \cite{MR4048444}. As a result, the proofs of computability and semicontinuity of dynamical degrees of maps \cite{xie2024recursive} do not apply in our setting, and these properties are not known for correspondences. Neither are there many well-understood examples. The only complete computations of $\lambda_1(f)$ for correspondences (beyond maps) are for monomial correspondences and Hurwitz correspondences \cite{MR4266360, MR4108910}.

Birational surface maps admit algebraically stable models, and their dynamical degrees are quadratic integers \cite{MR1867314}. On the other hand, there are rational maps of $\PP^2$ with no algebraically stable model \cite{MR2021001, MR4205407}. No general results on constructing improved models for surface correspondences are yet known.

\subsection*{Outline}
Section \ref{sect_prelim} contains background on correspondences and dynamical degrees.

In Section \ref{sect_model}, we recall and generalize the definition of algebraic billiards, introduce the Fermat hyperbola billiard, and construct an improved model $P$ of the domain. 

In Section \ref{sect_r}, we check that $P$ is a regular model for $r$. 

In Section \ref{sect_s}, we compute the part of the action of $s$ on $P$ that allows us to later show essential stability.

Section \ref{sect_dd}: We show that \rev{the lift $\hat{b}$ of the billiard} is essentially stable on $P$, Theorem \ref{thm_main_model} (\ref{it_thm_es}), and prove Theorem \ref{thm_main_dd} ($=$ Theorem \ref{thm_dd_body}).

Section \ref{sect_AS}: We compute indeterminacy of \rev{the lift $\hat{b}$} of \rev{the billiard} to $P$ and construct algebraically stable models when $\bk = \C$, proving Theorem \ref{thm_main_model} (\ref{it_thm_as}). Then we prove Theorem \ref{thm_main_ivrii}.

\rev{Note that Section \ref{sect_dd} and Section \ref{sect_AS} are independent, but both rely crucially on the computations in Section \ref{sect_r} and Section \ref{sect_s}.}

\subsection*{Acknowledgments}

We thank Richard Birkett, Laura DeMarco, Jeff Diller, and Alexey Glutsyuk for helpful conversations related to this work. Thanks to Anna Dietrich for invaluable assistance preparing the manuscript. The diagrams were prepared with Desmos and GeoGebra. This material is based upon work supported by the National Science Foundation under Award No. 2202752.

\section{Preliminaries} \label{sect_prelim}
This section contains background on correspondences over an algebraically closed field $\bk$. The facts are mostly well-known, but we have to introduce some nonstandard terminology (i.e. contracted curve vs. contraction). Background on divisors and intersection theory may be found in \cite{MR1644323, MR3617981}.

All our surfaces are smooth, projective, and irreducible.

\subsection{Correspondences} \label{sect_corrs}

\begin{defn} \label{def_corr}
Given irreducible projective varieties $X$ and $Y$ of the same dimension, a \emph{rational correspondence} $f = (\Gamma_f, X, Y)$, also written $f: X \vdash Y$, is an effective algebraic cycle in $X \times Y$ of the form
$$\Gamma_f = \sum_{i=1}^\nu m_i [\Gamma_i] \neq 0,$$
such that each summand satisfies $\dim \Gamma_i = \dim X$. We have projection maps 
$$\Pi_1 : X \times Y \to X, \quad \Pi_2 : X \times Y \to Y.$$
Let 
$$\pi_1: \Gamma_f \to X, \quad \pi_2 : \Gamma_f \to Y$$
be the restrictions of $\Pi_1, \Pi_2$ to the support of $\Gamma_f$. 
We write $f = (\Gamma_f, \pi_1, \pi_2)$ when it is helpful to remember the projections.

A rational correspondence is \emph{dominant} if the projections $\pi_1, \pi_2$ restrict to dominant, generically finite morphisms on each component $\Gamma_i$. For the rest of the paper, all correspondences are dominant.

Writing $m = \deg \pi_2$ and $n = \deg \pi_1$, we say that $f$ is an $m$-to-$n$ rational correspondence. We write
$$\deg_1 f \colonequals \deg \pi_1, \quad \deg_2 f \colonequals \deg \pi_2.$$

The \emph{adjoint} of $f : X \vdash Y$ is the correspondence $g : Y \vdash X$, defined by taking the image of $\Gamma_f$ in $Y \times X$ under $(x,y) \mapsto (y,x)$. \rev{We say $f$ is \emph{self-adjoint} if it is equal to its adjoint.}

The \emph{indeterminacy locus} and \emph{exceptional locus} are the varieties 
$$\Ind f = \{ x \in X : \dim \pi_1^{-1}(x) > 0\}, $$
$$\Exc f = \{ y \in Y : \dim \pi_2^{-1}(y) > 0 \}.$$

The rational correspondence $f$ is \emph{regular} if $\Ind f = \varnothing$. Then the projection $\pi_1$ is finite, rather than just generically finite. If further $\Exc f = \varnothing$, then $f$ is \emph{biregular}. 
\end{defn}

We are interested only in curve and surface correspondences. Dominant curve correspondences are biregular.

\begin{defn}
Let $f : X \vdash Y$ be a dominant correspondence of smooth, irreducible projective surfaces.

A \emph{contraction} of $f$ is an irreducible curve $V \subset \Gamma_f$ such that $\pi_2(V)$ is a point. A \emph{contracted curve} is an irreducible curve $W \subset X$ such that $\pi_1^{-1}(W)$ contains a contraction.

An \emph{expansion} of $f$ is an irreducible curve $V \subset \Gamma_f$ such that $\pi_1(V)$ is a point. An \emph{exceptional image curve} is a irreducible curve $W \subset Y$ such that $\pi_2^{-1}(W)$ contains an expansion.
\end{defn}

It is often useful to view correspondences as set-valued maps, as follows.
\begin{defn}
    The \emph{(total) image} of a set $\Sigma \subset X$ by $f$ is
$$f(\Sigma) \colonequals \pi_2 (\pi_1^{-1}(\Sigma)).$$
It is also useful to define a notion of image that avoids indeterminacy. The \emph{strict image} of a set $\Sigma\subset X$ by $f$, denoted $f_{\str}(\Sigma)$, is the Zariski closure in $Y$ of $f(\Sigma\smallsetminus \Ind f)$.
\end{defn}
Note that the strict image of a contracted curve may contain curves.

\begin{defn}
Suppose that $W \subset X$, $W' \subset Y$ are irreducible curves and $W'$ is a component of $f_{\str}(W)$. The \emph{dominant restriction of $f$}, denoted $f|_{W \vdash W'}$, is the dominant curve correspondence defined by the multiset of components of $\Gamma_f|_{W \times W'}$ that are dominant for both projections.
\end{defn}

The following lemma helps us \rev{to} compute images.
\begin{lemma} \label{lem_stein}
    Let $f : X \vdash Y$ be an $m$-to-$n$ rational correspondence of smooth projective surfaces, and let $p \in X$. If $f(p)$ contains a set $\Sigma$ of isolated points of total multiplicity $n$, then $f(p) = \Sigma$.
\end{lemma}

\begin{proof}
This is a consequence of Stein factorization applied to $\pi_1$ \cite[III.11.5]{Hartshorne}.
\end{proof}

\begin{defn} \label{def_composite}
We may form the \emph{composite}
$$g \circ f: X \vdash Z$$
of a pair of dominant rational correspondences
$$f : X \vdash Y, \quad g : Y \vdash Z.$$
It is defined by its graph, as follows.

The \emph{total composite graph} of $f$ and $g$ is the scheme-theoretic image of $\Gamma_f \times_Y \Gamma_g$ in $X \times Z$. It is a $2$-cycle. The graph $\Gamma_{g \circ f}$ is obtained by removing all components $\Gamma$ from the total composite graph on which either projection fails to be dominant.

The following construction is equivalent.
Let $V = Y \smallsetminus \Ind g$ and let $U = f^{-1}(V) \smallsetminus \Ind f$.
Let $\pi : U \times V \times Z \to U \times Z$ be the projection forgetting the $V$ factor.
Then define
$$\Gamma^{U \times Z}_{g \circ f} \colonequals \pi \left(\Gamma_f|_{U \times V} \times_V \Gamma_g|_{V \times Z} \right).$$
Then $\Gamma_{g \circ f}$ is defined as the Zariski closure of $\Gamma^{U \times Z}_{g \circ f}$ in $X \times Z$. We extend this definition linearly to $f$ and $g$ with graphs with non-reduced or multiple irreducible components.
\end{defn}

\subsection{Formal correspondences} \label{sect_formal}

Throughout this paper, it will be convenient to do local calculations in power series. To justify this simplification, we outline a theory of formal regular and rational correspondences of smooth surface germs. This is an extension of the theory of holomorphic or meromorphic fixed point germs \cite{MR2097722,MR2339287} to correspondences over arbitrary algebraically closed base fields. There are purely algebraic proofs over $\C$ that generalize without difficulty given the correct vocabulary. We therefore only give proofs for claims where there are differences between maps and correspondences.

For formal schemes, see \cite[II.9]{Hartshorne}. In this paper, all formal schemes are obtained from projective schemes by completion and restriction to closed formal subschemes. Let $X$ be a projective surface, let $x \in X$ be a smooth point of $X$, and let $G$ be the germ (formal neighborhood) of $X$ at $x$. It is isomorphic to the one-point formal scheme $(\A^2, 0)$ that has structure sheaf $\bk[[x_0, x_1]]$ at $x$, see \cite[II.9.3.4]{Hartshorne}.

\begin{defn}
    An \emph{iterated blowup} $\pi:G_\pi \to G$ is a formal scheme morphism obtained by a composition of finitely many point blowups
$$G_\pi = G_\ell \to \ldots \to G_0 = G.$$
It is naturally isomorphic to the formal neighborhood of the exceptional divisor upon blowing up the ambient surface of $G$ via $\pi$.

The \emph{exceptional curves} in $G_\pi$ are the exceptional image curves of $\pi^{-1}$. They generate a free abelian group $\Div G_\pi$. Iterated blowups form an inverse system for the relation $\trianglerighteq$ of domination, defined by
$$\pi' \trianglerighteq \pi \quad \Leftrightarrow \quad \exists \; \tilde{\pi} \colon \pi' = \pi \circ \tilde{\pi},$$
where $\tilde{\pi}$ is a composition of point blowups. If $G_\pi \neq G$, then the underlying scheme of $G_\pi$ is the union of the exceptional curves of $\pi$.
\end{defn}

\begin{defn}
    Let $G_1, G_2$ be iterated blowups of $G$. A \emph{formal rational $m$-to-$n$ correspondence} $f:G_1 \vdash G_2$ is given by its graph $\Gamma_f \subset G_1 \times G_2$, a formal effective $2$-cycle such that the two projections restricted to all sufficiently small infinitesimal neighborhoods approximating the scheme part of $G_1 \times G_2$ are generically finite of degrees \rev{$\deg_2 f = m$ and $\deg_1 f = n$} respectively.

All such correspondences in this paper are obtained by restricting correspondences on varieties as defined in Section \ref{sect_corrs}.
\end{defn}

\begin{defn}
\rev{Given a rational correspondence $G_\pi \vdash G_{\pi'}$ and} exceptional curves $W \subset G_\pi$, $W' \subset G_{\pi'}$ such that $W'$ is a component of $f_{\str}(W)$, the \emph{dominant restriction} $f|_{W \vdash W'}$ is the dominant curve correspondence defined by the multiset of components of $\Gamma_f|_{W \times W'}$ that are dominant for both projections.
\end{defn}

The following lemma is used in the proof of Lemma \ref{lem_c_test}.

\begin{lemma}\label{lem_corr_res}
    Say $f: G_\pi \vdash G_{\pi'}$. Then there exists $\pi''\trianglerighteq \pi'$ such that 
    \[f_{\pi\pi''} \colonequals (\pi'')^{-1}\circ \pi'\circ f\]
    has no contracted curves.
\end{lemma}
\begin{proof}
    There exist finitely many contractions of $\tau: \tilde{\Gamma}_f \to G_{\pi'}$, where $\tilde{\Gamma}_f$ is the desingularized graph of $f$. Call this number $\nu(f)$. Let $p\in G_{\pi'}$ be the image of a contraction of $\tau$. Let $\pi_0'\trianglerighteq \pi'$ be obtained by $\pi'_0 = \pi'\circ \pi_p$, where $\pi_p$ is the blowup at $p$. Let $E_p = \pi_p^{-1}(p)$. Then, by the universal property of blowing up, at least one contracted curve of $\tau$ is mapped onto $E_p$ by $\pi_p^{-1}\circ \tau$. Continuing in this fashion, after at most $\nu(f)$ point blowups, we obtain $\pi''\trianglerighteq \ldots \trianglerighteq \pi'_0 \trianglerighteq \pi'$ such that the map $\tilde{\Gamma}_{f_{\pi\pi''}} \to G_{\pi''}$ has no contracted curves. Hence, neither does $\Gamma_{f_{\pi\pi''}}\to G_{\pi''}$.
\end{proof}

\begin{remark}
    We cannot necessarily find $\pi''\trianglerighteq \pi$ and $\pi'''\trianglerighteq\pi'$ such that $f_{\pi''\pi'''}$ is biregular, unlike the case of rational maps. 
\end{remark}

The next lemma is a useful tool for ruling out contracted curves.

\begin{lemma}[Contracted curve test]\label{lem_c_test}
    Say $f: G_\pi \vdash G_{\pi'}$. Suppose that $\Delta$ is a prime divisor on $G_\pi$ such that 
    \[\sum_{\Delta'} \deg_1 f|_{\Delta \vdash \Delta'} = \deg_1 f,\]
    where $\Delta'$ ranges over \rev{$1$-dimensional} components of $f_{\str}(\Delta)$.
    Then $\Delta$ is not a contracted curve of $f$. 
\end{lemma}
\begin{proof}
    We prove the contrapositive. Let $f=(\Gamma_f,\tau_1,\tau_2)$. Say $\Delta\times p$ is a contraction. By Lemma \ref{lem_corr_res}, there exists $\pi''\trianglerighteq \pi'$ such that $f_{\pi\pi''}$ has no contractions, so the projection $\Gamma_{f_{\pi\pi''}}$ to $X_{\pi''}$ is finite. Let $\pi'' = \pi' \circ \tilde{\pi}$. Let $\bigcup_{i=1}^\ell \Delta''_i$ be the strict image of $\Delta$ in $X_{\pi''}$. For each $i$, if $\tilde{\pi} (\Delta''_i)$ is a curve, then 
    \[\deg_1 f_{\pi \pi'} |_{\Delta \vdash \tilde{\pi}(\Delta''_i)} = \deg_1 f_{\pi \pi ''}|_{\Delta \vdash \Delta''}.\]
    By the universal property of blowing up, the exceptional curve of the blowup at $p$ appears among the $\Delta''_i$.
    Therefore, 
    \[\sum_{\Delta'} \deg_1 f|_{\Delta \vdash \Delta'} < \sum_{\Delta''_i} \deg_1 f|_{\Delta \vdash \Delta''_i} = \deg_1 f.\]
\end{proof}

\subsection{Action on divisors} \label{sect_act_div}

\rev{
We now introduce key definitions for understanding degree growth of correspondences. Readers who are primarily interested in applications to the Ivrii Conjecture may skim this section.

\begin{remark} \label{rem_intuition}
For the benefit of readers with less background in algebraic geometry, we first give an informal introduction to these ideas. In complex geometry, given a dominant correspondence $f: \PP^1 \vdash \PP^1$, we may consider the image (or preimage) of a point as a multiset. For instance, if $f(z) = \pm \sqrt{z}$, the image of a general point consists of $2$ distinct elements, but $f(0)$ and $f(\infty)$ each consist of $1$ element, which we would like to consider as having multiplicity $2$. By counting this way, there is a well-defined pushforward homomorphism $f_*$ on $0$-dimensional homology classes, which in this example is just multiplication by $2$. Working with regular \emph{surface} correspondences, it is similarly helpful to consider the image (or preimage) of a \emph{curve} as a multiset of irreducible curves, i.e. a divisor; then there is an induced pushforward homomorphism $f_*$ on divisor numerical equivalence class groups, mimicking the induced map on $1$-complex-dimensional homology groups. Numerical equivalence class groups are free abelian groups of finite rank, so pushforwards can be written as matrices and studied with linear algebra.

The main conceptual difficulty is understanding the multiplicities. For example, the map $f: \PP^2 \to \PP^2$ defined by $[X:Y:Z] \mapsto [X^2:Y^2:Z^2]$ sends a general line to a conic with multiplicity $1$, but it sends the line $\Delta$ defined by $X = 0$ onto itself with multiplicity $2$. In the divisor group $\Div \PP^2$, we have $f_* \Delta = 2 \Delta$. Over $\C$, the multiplicity can be computed as the topological degree of the map $f|_\Delta$.

Working with rational surface correspondences, there is an additional technical difficulty: if a curve $\Delta$ passes through an indeterminacy point, it is not a priori clear whether the ``image of $\Delta$ with multiplicity'' should include the exceptional image curve or not. This is similar to the distinction between strict and total transform of a curve by a blowup. We therefore distinguish between $f_*$, which includes exceptional image curves, and a \emph{strict transform} $f_{\str}$, which does not.
\end{remark}
Now we introduce these concepts formally. As before, in this section $X$ is a smooth, projective, irreducible surface.
}

\begin{defn} \label{def_pfwd}
\rev{Let $f : X \vdash X$ be a dominant rational correspondence. Let $\Div X$ denote the group of Weil divisors on $X$. Let $\Num X$ denote $\Div X$ modulo numerical equivalence.}

    There are induced \emph{pushforward} homomorphisms 
$$f_*: \Div X \to \Div X,$$
$$f_*: \Num X \to \Num X,$$
\rev{defined as follows.}

    Let $u: V \to W$ be a generically finite and dominant morphism of smooth projective surfaces or formal surfaces, and let $\Delta$ be a prime divisor of $V$.

    We define $u_*: \Div V \to \Div W$ as follows.
    If $u(\Delta)$ is a point, then we define $u_* \Delta = 0$. \rev{Otherwise $u(\Delta)$ is a curve, and we define
    $$u_* \Delta = m u(\Delta),$$
    where the multiplicity $m$ is the degree of the function field extension induced by the dominant map of curves $u|_\Delta : \Delta \to u(\Delta)$. (Over $\C$, we have $m = \deg_{\Top} u|_\Delta$.)}
    
    We extend $u_*$ to all divisor classes by linearity. It descends to $u_*: \Num V \to \Num W$.

    There are also \emph{pullback} homomorphisms
    $$u^*: \Div W \to \Div V,$$ 
    $$u^*: \Num W \to \Num V.$$
    Rather than present a full definition, we just explain how to compute these in our setting. Let $\Delta$ be a prime divisor of $W$. If $\deg u = d$ and the inverse image of a general point of $\Delta$ consists of $d$ distinct points, then
    $$ u^* \Delta = u^{-1}(\Delta).$$
    If $f: X \vdash Y$, defined by $(\Gamma_f, \pi_1, \pi_2)$, is a correspondence with a smooth and irreducible graph, then the \emph{pushforward} homomorphism is
    $$f_* : \Num X \to \Num Y,$$
    $$f_* = (\pi_2)_* \circ (\pi_1)^*.$$
\end{defn}
This definition may be extended to graphs $\Gamma_f$ with singular components by working instead on a desingularized graph.

Pullback and pushforward are functorial for regular correspondences, but not for rational correspondences in general. The following criterion, stemming from work of Diller-Favre, describes when pushforward interacts well with composition \cite[Proposition 1.13]{MR1867314}.

\begin{lemma} \label{lem_df}
    Let $f : X \vdash Y$, $g: Y \vdash Z$ be dominant correspondences of smooth, irreducible projective surfaces. Let $f = (\Gamma_f, \pi_1, \pi_2).$ If, for every contraction $V$ of $f$, we have 
    $$\pi_2(V) \cap \Ind g = \varnothing,$$
    then $(g \circ f)_* = g_* \circ f_*$.
\end{lemma}

\begin{proof}
    The proof given in \cite[Proposition 1.4]{MR3342248} for rational maps over $\C$ adapts to our setting.
\end{proof}

\begin{defn} \label{def_str_div}
Let $f:X \vdash Y$. The \emph{strict transform} of a prime divisor $\Delta \in \Div X$ is the divisor
$$f_{\str} \Delta \colonequals \sum_{V} m_{V} V,$$
where $V$ ranges over curves in the \rev{strict image $f_{\str} (\Delta)$, i.e. the Zariski closure in $Y$ of $f(\Delta \smallsetminus \Ind f)$,} and $m_V$ is the weight of $V$ in $f_* \Delta$. This extends linearly to a homomorphism $f_{\str} : \Div X \to \Div Y$.
\end{defn}

\rev{
\begin{lemma} \label{lem_diff_support}
    Let $f:X \vdash Y$, and let $\Delta$ be a divisor. Then $f_*\Delta - f_{\str}\Delta$ is supported on exceptional image curves.
\end{lemma}
\begin{proof}
    It suffices to prove the claim when $\Delta$ is a prime divisor, since then we can extend linearly to all divisors. The support of $f_* \Delta$ is contained in the total image $f(\Delta)$. The curves in $f(\Delta)$ are either exceptional image curves in $f(\Ind f)$, or are in the Zariski closure of $f(\Delta \smallsetminus \Ind f)$ in $Y$, which is the strict image $f_{\str}(\Delta)$. Therefore, by definition of the multiplicities $m_V$ in Definition \ref{def_str_div}, the difference consists only of exceptional image curves.
\end{proof}
}

\subsection{Dynamical degrees}
In this section, we assume that $f : X \vdash X$ is a dominant rational correspondence on an irreducible smooth projective surface $X$ over $\bk$. 

\begin{defn}[\cite{MR4048444}] \label{def_dd}
Let $\Delta \in \Num X$ be an ample class.
The \emph{(first) dynamical degree} of $f$, denoted $\lambda_1(f)$, is the quantity
\begin{equation} \label{eq_dd_norm_def}
    \lambda_1(f) \colonequals \lim_{n \to \infty} ((f^n)_* \Delta \cdot \Delta)^{1/n}.
\end{equation}
The limit \eqref{eq_dd_norm_def} exists and is independent of the choice of $\Delta$.
\end{defn}

\begin{defn} \label{def_as}
We say $f: X \vdash X$ is \emph{algebraically stable} if $(f_*)^n = (f^n)_*$ on $\Num X$ for all $n \in \N$.
\end{defn}

The \emph{spectral radius} of a linear endomorphism $T$ is $\rad T \colonequals \max_\lambda \abs{\lambda},$ where $\lambda$ ranges over the eigenvalues of $T$. \rev{If a correspondence $f$ is algebraically stable, then $\lambda_1(f) = \rad f_*$. See Lemma \ref{lem_AAS}.}

\rev{Next, we introduce some tools for understanding dynamical degrees, both of individual correspondences and correspondences in families. A family of smooth, irreducible, projective surfaces parametrized by a scheme $S$ can be described as a projective, surjective map $\pi: X \to S$ that is smooth of relative dimension $2$. This last condition is an algebraic version of the notion of local submersion from differential geometry, and it ensures that the fibers are smooth.}

\begin{thm}[standard properties of dynamical degrees] \label{thm_dd_props}
\leavevmode
\begin{enumerate}
    \item \label{it_birat_invariant} The dynamical degree $\lambda_1(f)$ is a birational conjugacy invariant of the pair $(X, f)$. That is, if $g : Y \dashrightarrow X$ is a birational map, then
    $$\lambda_1(f) = \lambda_1 (g^{-1} \circ f \circ g).$$
    \item \label{it_first_iterate} We have
    $$\lambda_1(f) \leq \rad (f_* : \Num X \to \Num X).$$
    \item \label{it_as} If for all $n \geq 0$, we have
    $$ f^{n}(\Exc f) \cap \Ind f = \varnothing,$$
    then $f$ is algebraically stable.
    \item \label{it_specialization} Let $S$ be a smooth, integral scheme with generic point $\eta$. Let $\pi : X \to S$ be a \rev{projective, surjective morphism that is smooth of} relative dimension $2$. Let $X_p$ denote the fiber at $p \in S$. Let $f: X \vdash X$ be a rational correspondence \rev{such that $\pi = \pi \circ f$}, with the property that, for each $p \in S$, the restriction $\Gamma_f|_{X_p \times X_p}$ is the graph of a dominant rational correspondence $f_p : X_p \vdash X_p$. Then, for all $p \in S$, we have
        $$ \lambda_1 ( f_p ) \leq \lambda_1 (f_\eta).$$
\end{enumerate}
\end{thm}
\begin{proof}
Claim (\ref{it_birat_invariant}) is a deep theorem of Truong \cite[Theorem 1.1]{MR4048444}. Claim (\ref{it_first_iterate}) is \cite[Theorem 2.18 (2)]{billiardsI}.
Claim (\ref{it_as}) follows from iterating Lemma \ref{lem_df}.
Claim (\ref{it_specialization}) is a straightforward generalization of \cite[Lemma 4.1]{MR3332894}.
\end{proof}

\rev{Recall from Definition \ref{def_es} that a dominant surface correspondence $f: X \vdash X$ is} \emph{essentially stable} if there exists a big-and-nef divisor $\Delta$ on $X$ such that, for all $n \in \N$,
    \begin{equation}
    \left(f^n\right)_*\Delta\cdot \Delta  = (f_*)^n \Delta \cdot \Delta. 
    \label{eq_def_es}
    \end{equation}

This property is a weakening of algebraic stability that still permits computation of the dynamical degree, as shown by the following lemma.

\begin{lemma} \label{lem_AAS}
    If $f: X \vdash X$ is essentially stable, then
    $$ \lambda_1(f) = \rad f_*. $$
\end{lemma}
\begin{proof}
    Let $\Delta$ be a big-and-nef divisor satisfying \eqref{eq_def_es}. We claim the following circle of inequalities holds:
    \begin{align*}
        \lambda_1(f) & \geq 
        \lim_{n\to \infty}\left(\left(f^n\right)_*\Delta\cdot \Delta\right)^{1/n} \\
        &= \lim_{n\to \infty} \left(\left(f_*\right)^n\Delta\cdot \Delta\right)^{1/n} \\
        &= \rad f_* \\
        & \geq \lambda_1(f).
    \end{align*}
    These estimates are all standard. The first line holds because any divisor is a difference of ample divisors. The second line is \eqref{eq_def_es}. To prove the third line, note that big-and-nef classes are in the interior of the pseudoeffective cone \cite{MR2095471}, and $f_*$ maps the pseudoeffective cone into itself; then apply the Perron-Frobenius theorem for cone-preserving operators. The last line is Theorem \ref{thm_dd_props} (\ref{it_first_iterate}).
\end{proof}

\section{Billiards and blowups} \label{sect_model}

\subsection{Billiards in algebraic curves} \label{sect_setup}
In this subsection, we recall the definition of the algebraic billiards correspondence, following \cite{billiardsI}. Our presentation here is more general in that we allow reflection with respect to any choice of non-degenerate quadratic form $\Theta$, which allows us to reason about changes of coordinates. We limit attention to the case of smooth plane curves over a field $\bk$ of characteristic not equal to $2$.

Let $X = \PP^2$ with coordinates $[X_0 : X_1 : X_2]$, and $Q = \PP^2$ with coordinates $[Q_0 : Q_1 : Q_2]$. We introduce affine coordinates
\begin{equation}
x_0 = \frac{X_0}{X_2}, \quad x_1 = \frac{X_1}{X_2}, \quad q_0 = \frac{Q_0}{Q_2}, \quad q_1 =\frac{Q_1}{Q_2}. \label{eq_aff_coords}
\end{equation}
These coordinates define the \emph{main affine subsets}
$$ \A^2 = \Spec \bk[x_0, x_1] \subset X, \quad \A^2 = \Spec \bk[q_0, q_1] \subset Q.$$
Having chosen these affines, there are distinguished lines at infinity 
$$X_\infty \colonequals X \smallsetminus \A^2, \quad Q_\infty \colonequals Q \smallsetminus \A^2.$$
We view the space $\A^4 = \Spec \bk[x_0, x_1, q_0, q_1]$ as the tangent space $T\A^2$, where $\A^2 = \Spec \bk[x_0, x_1]$ and the cotangent vectors $q_0 = dx_0, q_1 = dx_1$ are treated as formal generators of a symmetric algebra. The coordinates $q_0, q_1$ identify all the individual tangent spaces and are translation-invariant. The compactification $X \times Q$ has the nice feature that there is a canonical identification of both $X_\infty$ and $Q_\infty$ with the space $\PP^1$ of slopes of lines in $\A^2$.

Fix $d \geq 2$. Let $C \subset X$ be a smooth, irreducible curve of degree $d$, defined by some affine equation $f_C(x_0, x_1)= 0$.

Let $\Theta : T_0 \A^2 \to \bk$ be a non-degenerate quadratic form, and let $D$ be the Zariski closure of $\Theta^{-1}(1)$ in $Q$. We call $X \times D$ the \emph{space of unit tangent vectors with respect to $\Theta$}, although properly speaking, it is a compactification thereof. Then $C \times D$ is (a compactification of) the space of unit tangent vectors for $\Theta$ based on $C \cap \A^2$. 

Points of $D$ are called \emph{directions}. There is a $2$-to-$1$ ramified cover, the \emph{slope map}, defined by
$$[ \cdot ] : D \to \PP^1,$$
$$[Q_0 : Q_1 : Q_2] \mapsto [Q_0 : Q_1].$$

Let
$$C_\infty \colonequals C \cap X_\infty,$$
$$D_\infty \colonequals D \cap Q_\infty.$$
Points in $D_\infty$ are called \emph{isotropic directions}. Their images in $\PP^1$ are called \emph{isotropic slopes}. 
Since $\charac k \neq 2$, any point $x \in C$ with non-isotropic tangent slope has $2$ tangent directions, with the same slope.

The billiards correspondence $b = b_{C,D}$ associated to the pair $(C,D)$ is a rational correspondence
$$b : C \times D \vdash C \times D,$$
$$ b = r \circ s,$$
where $s = s_{C,D}$ is secant and $r = r_{C,D}$ is reflection, defined as follows.

\begin{defn}[secant] \label{def_s}
We shall define $\Gamma_s \subset (C \times D)^2$. We denote the first and second factors of this product by $C \times D$ and $C' \times D'$, respectively, and similarly add a prime to denote any quantity that is defined in the second factor.

We define a relation $S_+$ on $X \times Q$, via its graph, as the subvariety $X \times Q \times X' \times Q'$ cut out by the equations
\begin{align}
    &q = q', \label{eq_splus_1} \\
    &\begin{vmatrix}
    X_0 & X'_0 & Q_0 \\
    X_1 & X'_1 & Q_1 \\
    X_2 & X'_2 & 0
    \end{vmatrix} = 0. \label{eq_splus_2}
\end{align}
Restricting $S_+$ to $C \times D \times C' \times D'$ gives us a relation $s_{+}$ on $C \times D$. Notice that the graph $\Gamma_{s_+}$ contains the graph $\Gamma_1$ of the identity map. The \emph{secant correspondence}
$$s : C \times D \vdash C \times D$$
is the variety with graph 
$$\Gamma_{s} = \Gamma_{s_{+}} - \Gamma_1.$$
Then $s$ is a dominant $(d-1)$-to-$(d-1)$ rational correspondence. To check that the variety defined by these equations defines a dominant correspondence, one must show every component is a surface projecting surjectively to each factor; we prove this in \cite[Proposition 3.8]{billiardsI} with the assumption that $D$ is the Euclidean form, and the general case is similar.
\end{defn}

\begin{defn}[reflection] \label{def_r} 
 Given a non-isotropic slope $m \in \PP^1$, the \emph{reflection by $m$} is the unique map $r_m: D \to D$
that fixes the two directions of slope $m$ and exchanges the $2$ points of $D_\infty$. Such a map exists because $D$ is centrally symmetric, and it is unique because $D$ is abstractly isomorphic to $\PP^1$ and thus has a sharply $3$-transitive automorphism group. Since $r_m^2$ fixes at least $3$ points, $r_m^2$ is the identity, so $r_m$ is an involution.

Let $t : C \to \PP^1$ be the tangent slope map, sending $x \in C \smallsetminus C_\infty$ to 
$$t(x) = \left[ -\frac{\partial f_C}{\partial x_1}(x): \frac{\partial f_C}{\partial x_0} (x) \right].$$
This is defined because $C$ was assumed smooth. 

The \emph{reflection map}
$$r_{C,D}: C \times D \dashrightarrow C \times D,$$
is defined for $x \in C \smallsetminus C_\infty$, $v \in D \smallsetminus D_\infty$ by 
$$ (x,v) \mapsto (x, r_{t(x)} v).$$
\end{defn}

\rev{Definition \ref{def_r} is based on the observation of Glutsyuk that the billiards reflection law on the projectivized tangent space permutes isotropic directions \cite{MR3224419}.}

\begin{defn}
    The \emph{billiards correspondence} $b = b_{C, D}$ associated to $(C,D)$ is
$$b_{C,D} : C \times D \vdash C \times D,$$
$$b_{C,D} \colonequals r_{C,D} \circ s_{C,D}.$$
\end{defn}
Since $s_{C,D}$ is $(d-1)$-to-$(d-1)$ and $r_{C,D}$ is $1$-to-$1$, the composite $b_{C,D}$ is $(d-1)$-to-$(d-1)$.

For applications to classical billiards, the obvious choice of quadratic form $\Theta$ is $q_0^2 + q_1^2$; then orbits of the classical billiard map in the real locus are also orbits of $b$. However, it can also be helpful to change coordinates, as follows.

Consider $\A^2 = X \smallsetminus X_\infty$. Let $J \in \Aut(X)$ be a projective transformation such that $J(X_\infty) = X_\infty$. Then $J$ induces a map $T\A^2 \to T\A^2$, linear on each tangent space, and $J_* : T_x \A^2 \to T_{J(x)}\A^2$ is independent of $x$ after identifying $T_x \A^2$ and $T_{J(x)} \A^2$ with $T_0 \A^2$ by translation. In other words, the differential of $J$ has the same matrix in the canonical coordinates $q_0, q_1$ of each tangent space. Hence there is a well-defined map $J_*: Q \to Q$. We note that if $J$ fixes the origin, then it is linear, and the matrix of $J_*$ in basis $q_0, q_1$ agrees with that of $J$ in $x_0, x_1$.

\begin{lemma}[changing coordinates] \label{lemma_COC}
 Let $J_* : Q \to Q$ be the pushforward induced on tangent vectors. Then $J \times J_*$ conjugates $s_{C,D}$ to $s_{J(C), J_* (D)}$, $r_{C,D}$ to $r_{J(C), J_* (D)}$, and $b_{C,D}$ to $b_{J(C), J_* (D)}$. In particular, if $J(C) = C$ and $J_*(D) = D$, then $J \times J_*$ is an automorphism of $s_{C,D}$, $r_{C,D}$, and $b_{C, D}$.
\end{lemma}
\begin{proof}
To prove the claims about $b$, it suffices to prove the claims about $s$ and $r$. To check that $s_{C,D}$ conjugates to $s_{J(C), J_* (D)}$, note that $J$ and $J_*$ have the same action on the space of slopes $\PP^1$, so $J \times J_*$ preserves \eqref{eq_splus_1} and \eqref{eq_splus_2}. To check $r$, we use that $J_*(D)$ is again the unit tangent space of some nondegenerate quadratic form, and that $J_*(D)$ takes $D_\infty$ to $J_*(D)_\infty$ and respects the tangent slope map $t$, being a differential.
\end{proof}

From now on, we write $s = s_{C,D}$, $r = r_{C,D}$, $b = b_{C,D}$ when the underlying pair $C,D$ is fixed.

The next two lemmas collect some basic facts about the correspondences $s$ and $r$ that we proved in our previous work.

\begin{lemma}[basic properties of $s$] \label{lem_s_basic}
    Let $C$ be a smooth curve in $\PP^2_{\bk}$ of degree $d \geq 2$, and let $D$ be the unit tangent space of a nondegenerate quadratic form $\Theta$. Assume $\charac \bk \neq 2$.
    \begin{enumerate}
        \item The relation $s$ defines a self-adjoint $(d-1)$-to-$(d-1)$ correspondence on $C \times D$.
        \rev{For any $(x, v) \in C \times D$ not in $\Ind s$, let $\ell$ be the line through $x$ of slope $v$, and consider $C \cap \ell$ as a multiset of total multiplicity $d$. Then $C \cap \ell$ contains $x$ with positive multiplicity, so the multiset $(C \cap \ell) - x$ obtained by lowering the multiplicity of $x$ by $1$ is defined. The image $s(x,v)$ is the multiset of total multiplicity $d-1$,
        $$s(x,v) = ((C \cap \ell) - x) \times v.$$
        }
        \item The indeterminacy locus of $s$ is
$$\Ind s = \Exc s = \{ (x_\infty, v) : \; x_\infty \in C_\infty, \; v \in D, \; [v] = t(x_\infty) \}.$$
    \item For each $p = (x_\infty, v) \in \Ind s$, let $C^{(p)} \colonequals C \times v$. The set of contractions of $s$ is
    $$\{ C^{(p)} \times p : p \in \Ind s\}.$$
    \end{enumerate}
\end{lemma}

\begin{proof}
    The following proofs were stated for the Euclidean form, and the general case follows by changing coordinates using Lemma \ref{lemma_COC}.
    \begin{enumerate}
        \item Special case of \cite[Proposition 3.8(1)]{billiardsI}.
        \item \cite[Proposition 3.8(3, 4)]{billiardsI}, noting that $\Ind s = \Exc s$ by self-adjointness.
        \item \cite[Proposition 3.8(3, 4)]{billiardsI}.
    \end{enumerate}
\end{proof}

\begin{lemma}[basic properties of $r$] \label{lem_r_basic}
    Let $C$ be a smooth curve in $\PP^2_{\bk}$ of degree $d \geq 2$, and let $D$ be the unit tangent space of a nondegenerate quadratic form $\Theta$. Assume $\charac \bk \neq 2$.
    \begin{enumerate}
    \item The relation $r$ defines a self-adjoint $1$-to-$1$ rational correspondence on $C \times D$.
    \item The indeterminacy locus of $r$ is
$$\Ind r = \Exc r = \{ (x, v_\infty) : \; x \in C, \; v_\infty \in D_\infty, \; [v_\infty] = t(x) \}.$$
    \item For each $p = (x, v_\infty) \in \Ind r$, let $D^{(p)} \colonequals x \times D$. The set of contractions of $r$ is
    $$ \{ D^{(p)} \times p : p \in \Ind r \}.$$
    \end{enumerate}
\end{lemma}

\begin{proof}
The following proofs were stated for the Euclidean form, and the general case follows by changing coordinates using Lemma \ref{lemma_COC}.
\begin{enumerate}
    \item \cite[Proposition 3.13 (1)]{billiardsI}.
    \item \cite[Proposition 3.13 (3,4)]{billiardsI}; note $\Ind r = \Exc r$ by self-adjointness.
    \item \cite[Proposition 3.13 (3,4)]{billiardsI}.
\end{enumerate}
\end{proof}

\subsection{The Fermat hyperbola} \label{sect_fh}
Now we specialize to our curve of interest. For the rest of the paper, let $\bk$ be an algebraically closed field such that \rev{$\charac \bk = 0$ or} $\gcd(\charac \bk, 2d) = 1$.

The \emph{Fermat hyperbola} of degree $d \geq 2$ is the projective plane curve
$$C : \; (X_0 - iX_1)^d + (X_0 + iX_1)^d = X_2^d.$$
The standard choice of quadratic form for algebraic billiards is the so-called Euclidean form $\Theta = q_0^2 + q_1^2$, which has space of unit tangent vectors
$$D: \; Q_0^2 + Q_1^2 = Q_2^2.$$
\rev{The key property of the pair $(C,D)$ is that, at each point $x \in C$ with isotropic tangent slope, the tangent line to $C$ at $x$ has no other intersections with $C$.} This is easiest to see after changing coordinates to the standard Fermat curve billiard with quadratic form $y_0 y_1$. Specifically, the projective transformation
\begin{equation} \label{eq_J}
J : \PP^2 \to \PP^2,
\end{equation}
$$ [X_0 : X_1 : X_2] \mapsto [X_0 - iX_1 : X_0 + iX_1 : X_2],$$
takes $C, D$ to
$$J(C) :\;  Z_0^d + Z_1^d = Z_2^d,$$
$$J_*(D) : \; Y_0 Y_1 = Y_2^2.$$

By Lemma \ref{lemma_COC}, the map $J \times J_*$ is a conjugacy from $b_{C, D}$ to $b_{J(C), J_* (D)}$. Computing with $b_{J(C), J_* (D)}$ is easier because the isotropic directions are $0$ and $\infty$ rather than $\pm i$.

\begin{lemma}[geometry of the Fermat hyperbola] \label{lem_fh_geo} \leavevmode
\begin{enumerate}
    \item 
The Fermat hyperbola $C$ is a smooth curve in $\PP^2$.
\item There are $d$ points in $C_\infty$. If $x_\infty = [X_0 : X_1 : 0 ] \in C_\infty$, then $t(x_\infty) = [X_0 : X_1]$.
\item The points $[1 : \pm i : 0]$ are not in $C_\infty$. \label{it_non_isot}
\item The points of $C$ with tangent slope $\pm i$ are precisely the points
$$\left\{ \left[ \frac{\zeta}{2} : \mp \frac{i\zeta}{2} : 1 \right] : \zeta^d = 1 \right\}.$$
At each of these points, the tangent line to $C$ \rev{has maximum possible tangency order, and thus has no other intersections with $C$.} \label{it_no_other_intersections}
\end{enumerate}
\end{lemma}

\begin{proof} \leavevmode
\begin{enumerate}
    \item Since \rev{$d \neq 0$ in $\bk$,} the Jacobian criterion for $J(C)$ shows smoothness.
    \item Notice that $J(C)$ has $d$ points at infinity and $J(X_\infty) = X_\infty$. The claim about $t(x_\infty)$ is clear because the intersections $C \cap X_\infty$ are transversal.
    \item These points do not satisfy the defining equation of $C$.
    \item Since $J([1:i:0]) = [1:0:0]$ and $J([1:-i:0]) = [0:1:0]$, the transformation $J$ takes slope $i$ to $0$ and slope $-i$ to $\infty$. On $J(C)$, the points of tangent slope $0$ are exactly those of the form $[0 : \zeta : 1]$, and there are no other points on $\Phi$ of the form $[Z_0 : \zeta : 1]$. Since $J^{-1}([0 : \zeta : 1]) = [\zeta: - i \zeta : 2]$, this proves the claim for slope $i$. The claim for $-i$ follows by Galois.
\end{enumerate}
\end{proof}

An \emph{exceptional point} of a correspondence $f: V \vdash V$ is a point $p \in V$ such that $f^{-1}(p) = \{p\}$. The property that we need concerning the Fermat hyperbola is that the indeterminacy points of $s$ and $r$ are exceptional for each other.

\begin{lemma}[exceptionality] \label{lem_e}
Let $C$ be the Fermat hyperbola, and consider $s, r : C \times D \vdash C \times D.$
\begin{enumerate}
\item There are $2d$ points in $\Ind s$.
\item There are $2d$ points in $\Ind r$, namely the points of the form
$$ \left(\left[ \frac{\zeta}{2} : \mp \frac{i\zeta}{2} : 1 \right], \left[1 : \pm i : 0 \right] \right)$$
for each $d$-th root of unity $\zeta$. \label{it_ind_r}
\item $\Ind s \cap \Ind r = \varnothing$.
\item If $p \in \Ind s$, then $r^{-1}(p) = \{p\}$.
\item If $p \in \Ind r$, then $s^{-1}(p) = \{p\}$. \label{it_e_ind_r}
\end{enumerate}
\end{lemma}

\begin{proof} \leavevmode
    \begin{enumerate}
        \item By Lemma \ref{lem_fh_geo}, there are $d$ points in $C_\infty$ and none of them are $[1:\pm i : 0]$, so each contributes $2$ points in $\Ind s$ by Lemma \ref{lem_s_basic}.
        \item Lemma \ref{lem_r_basic} describes $\Ind r$ in terms of the points of $C$ of isotropic tangent slope, and we computed these in Lemma \ref{lem_fh_geo} (\ref{it_no_other_intersections}).
        \item By Lemma \ref{lem_s_basic} and Lemma \ref{lem_r_basic}, if $p \in \Ind r \cap \Ind s$, then $p = (x_\infty, v_\infty)$ where $x_\infty \in C_\infty$ and $v_\infty \in D_\infty$, but this is false by \eqref{it_ind_r}. 
        \item By Lemma \ref{lem_s_basic}, if $p \in \Ind s$, then $p = (x_\infty,v)$ where $x_\infty \in C_\infty$ and $[v] = t(x_\infty)$. By Lemma \ref{lem_fh_geo} (\ref{it_non_isot}), $x_\infty \neq [1 : \pm i : 0]$, so since $x_\infty \in C_\infty$, the slope $t(x_\infty)$ is non-isotropic. The reflection $r_x$ at any point $x \in C$ of non-isotropic tangent slope fixes the tangent directions to $C$ at $x$, so $r$ fixes $(x_\infty,v)$.
        \item By Lemma \ref{lem_r_basic}, if $p \in \Ind r$, then $p = (x, v_\infty)$ where $v_\infty \in D_\infty$ and $[v_\infty] = t(x)$. Then
        \begin{align*}
            s^{-1}(p) &= s(p) & \text{(self-adjointness)}\\
            &= ((C \cap \ell(x, v_\infty)) \smallsetminus \{x\}) \times v_\infty & \text{(Lemma \ref{lem_s_basic})} \\
            &= \{(x, v_\infty)\} & \text{(as sets, by Lemma \ref{lem_fh_geo} (\ref{it_no_other_intersections}))} \\
            &= \{p\}.
        \end{align*}
    \end{enumerate}
\end{proof}

In classical billiards over $\R$, tables with rotational or reflective symmetry give rise to billiard maps with automorphisms. This is because rotations and reflections belong to the orthogonal group $O_2(\R)$, which is the group of linear transformations preserving the Euclidean metric. The Fermat hyperbola and its billiard have symmetry over $\R$: rotation by $2\pi / d$ and reflection across the horizontal axis, as suggested by Figure \ref{fig_fh}. Over $\bk$, there are also symmetries, as shown by the next lemma.

\begin{lemma}[automorphisms] \label{lemma_aut}
The correspondences $s_{C,D}$, $r_{C,D}$, and $b_{C,D}$ share a group of automorphisms that acts transitively on $\Ind r$.
\end{lemma}

\begin{proof}
Let $\zeta$ denote a primitive $d$-th root of unity in $\bk$. Let $\xi, \phi \colon \PP^2 \to \PP^2$ be the automorphisms given by the $\PGL_3$-elements
\[
\xi = \begin{bmatrix}
    \frac{\zeta + \zeta^{-1}}{2}  &  -\frac{\zeta - \zeta^{-1}}{2i}  &  0 \\
    \frac{\zeta - \zeta^{-1}}{2i}  &  \frac{\zeta + \zeta^{-1}}{2}  &  0 \\
    0 & 0 & 1
\end{bmatrix},\quad
\phi = \begin{bmatrix}
    1&0&0\\
     0&-1&0\\
     0&0&1
\end{bmatrix}.
\]
Both maps preserve $X_\infty$, and $\xi(C) = C$, $\xi_*(D)$, $\phi(C) = C$, $\phi_*(D) = D$.
Then by Lemma \ref{lemma_COC}, $\xi \times \xi_*$ and $\phi \times \phi_*$ are automorphisms of the surface $C \times D$ and of the correspondences $s_{C,D}$, $r_{C,D}$, and $b_{C,D}$. Transitivity is easy to check from Lemma \ref{lem_e} (\ref{it_ind_r}).
\end{proof}

Over $\C$, taking $\zeta = e^{2\pi i /d}$, the matrix $\xi$ is the real rotation by $2 \pi / d$, by the usual cosine and sine formulas, and $\phi$ is a reflection.

\subsection{Local formulas at $\Ind r$} \label{sect_series}

We continue with the Fermat hyperbola, defined as 
$$C = \{(X_0 - iX_1)^d + (X_0 + iX_1)^d = X_2^d\}.$$

To guess the blowup that regularizes $r$ around a point $p \in \Ind r$, we start by finding local formulas in Laurent series for $r$ and $s$ around $p$. Because of the existence of automorphisms acting transitively on $\Ind r$, the formulas are independent of the choice of $p$. Once we have these formulas, we can forget the geometric origin of $r$ and $s$ until we return to global computations in Section \ref{sect_dd}.

Let
$$p_0 \colonequals \left( \left[ \frac{1}{2}:-\frac{i}{2}:1\right] , [ 1:i:0 ] \right) \in \Ind r.$$

We work with formal neighborhoods, as outlined in Section \ref{sect_formal}. Let $(C \times D, p_0)$ denote the formal neighborhood of $C \times D$ about $p_0$. We coordinatize $(C \times D, p_0)$ by applying the transformation $J$, see \eqref{eq_J}, then translating $J(\left[ \frac{1}{2}:-\frac{i}{2}:1\right]) = [0: 1 : 1]$ to the origin $[0:0:1]$. Explicitly, define
$$J^1 : \PP^2 \to \PP^2,$$
$$[X_0 : X_1 : X_2] \mapsto [X_0 - iX_1: X_0 + i X_1 - X_2: X_2].$$
By Lemma \ref{lemma_COC}, there is a conjugacy from $b_{C, D}$ to $b_{J^1(C), J_*^1(D)}$.

\rev{Recall that we defined affine coordinates $x_0,x_1,q_0,q_1$ in \eqref{eq_aff_coords}.} Let 
$$ z_0 \colonequals x_0 - ix_1, \quad z_1 \colonequals x_0 + ix_1 - 1, $$
$$y_0 \colonequals q_0 - i q_1 , \quad y_1 \colonequals q_0 + iq_1.$$
In the coordinates $(z_0, z_1, y_0, y_1)$, the affine equation for $J^1(C)$ is
\begin{equation}\label{eq_C}
    z_0^d + (z_1+1)^d = 1,
\end{equation}
and the affine equation for $J^1_*(D)$ is
\begin{equation}\label{eq_D}
    y_0y_1 = 1.
\end{equation}

\rev{We will work in local coordinates
    \begin{equation} \label{eq_yz}
        y \colonequals y_1,\quad z \colonequals z_0.
    \end{equation}
The new notation is just to avoid writing subscripts throughout the paper. Geometrically, the function $z$ is a local parameter on $J(C)$ near the chosen point of isotropic tangent slope, and the function $y$ is a square root of the slope of the tangent vector in $J^*(D)$, by \eqref{eq_D}.}

\begin{lemma} \label{lemma_yz}
    The functions \rev{$y,z$} define a formal isomorphism from the germ $(C \times D, p_0)$ to $(\A^2, 0)$, the formal scheme with structure sheaf $\Spec \bk[[y,z]]$.
\end{lemma}
\begin{proof}
The image of $p_0$ in $J^1(C) \times J^1_*(D)$ is $p_1 \colonequals ([0:0:1],[1:0:0])$. Our chosen formal coordinate system on $J^1(C) \times J^1_*(D)$ about $p_1$ will be $(y_1, z_0)$. To justify that these functions really do provide a local formal system of coordinates, notice that $J^1(C)$ is smooth and $z_0$ is a uniformizer at $[0:0:1]$, and that $y_1$ defines an isomorphism $J^1_*(D) \cong \PP^1$ taking $[1:0:0]$ to $0$.
\end{proof}

Our next step is to write local expressions for reflection and secant.

\begin{prop}[local formulas] \label{prop_series}
    Let $p \in \Ind r$. There are formal coordinates $y, z$ identifying the germ $(C \times D,p)$ with $(\A^2,0)$ in which the following formulas hold.
    \begin{enumerate}
        \item \label{item_prop_series_r}
        The reflection correspondence $r_{C,D}$ defines a formal rational involution
        \[r: (\A^2,0) \dashrightarrow (\A^2,0),\]
        of the form 
        \[(y,z)\mapsto \left(\frac{1}{y} \left( z^{d-1}(1+ O(z^d))\right), z\right),\]
        where $O(z^d)$ denotes a power series in $z$ that vanishes at $0$ to order at least $d$.
        \item \label{item_prop_series_s}
        The secant correspondence $s_{C,D}$ defines a $(d-1)$-to-$(d-1)$ formal regular correspondence 
        $$s: (\A^2,0)\vdash (\A^2,0)$$
        with graph defined by local equations of the form 
        \[y' = y, \]
        \[\rev{-dy^2 = z^{d-1}+z^{d-2}z' + \ldots + (z')^{d-1}+A(z,z'),}\]
        where $A\in \bk[[z,z']]$ is in the ideal $\langle z, z'\rangle^d$.
    \end{enumerate}
\end{prop}
\begin{proof}
We reduce to the case $p = p_0$ and $(y,z) = (y_1, z_0)$ as follows. Choose some simultaneous automorphism $\alpha$ of $s_{C,D}, r_{C,D}$ moving $p$ to $p_0$; this exists by Lemma \ref{lemma_aut}. Then take $(y,z)$ = $(y_1 \circ \alpha, z_0 \circ \alpha)$ as defined in Lemma \ref{lemma_yz}.

By assumption on characteristic, on $J^1(C)$, we may express $z_1$ as a power series $B(z_0) \in \bk[[z_0]]$, namely 
    \begin{equation}
        z_1 = B(z_0) \colonequals -\frac{1}{d}z_0^d + O(z_0^{d+1}), \label{eq_z_series}
    \end{equation}
    We also may express $y_0$ in the obvious way as a function of $y_1$,
    \[y_0 = \frac{1}{y_1}.\]   
        \par\enspace(1) We derive the formula for $r$ first. Since $J^1 \times J_*^1$ conjugates $r_{C,D}$ to $r_{J^1(C), J^1_*(D)}$ by Lemma \ref{lemma_COC}, we just need to express $r_{J^1(C), J^1_*(D)}$ in terms of $y,z$.  Clearly $r_{J^1(C), J^1_*(D)}$ preserves $z_0$ and $z_1$, so $r$ preserves $z$. And $r_{J^1(C), J^1_*(D)}$ exchanges the two points of $J^1_*(D)_\infty$ while fixing the two points on $J^1_*(D)$ of slope $t(z_0,z_1) = [t_0:t_1]$, whenever these four points are distinct, by definition of reflection. The fixed points are given by $y_1 = \pm \sqrt{t_0/t_1}$. The unique such involution is 
        \[(y_0, y_1, z_0, z_1) \mapsto \left(-\frac{t_0 y_1} {t_1}, -\frac{t_1}{t_0 y_1}, z_0, z_1\right).\]
        Thus, viewed as a meromorphic map on $(\A^2,0)$, the formula for $r$ is 
        \[(y,z)\mapsto \left(-\frac{t_1}{t_0 y}, z\right),\]
        where by assumption on $\charac \bk$ we have 
        \begin{align*}
            t(z_0,z_1) &=\left[  -\frac{\partial}{\partial z_1}(z_0^d +(z_1+1)^d -1) :  \frac{\partial}{\partial z_0}( z_0^d + (z_1 + 1)^d - 1) \right]&\\
            &=\left[-d (z_1+1)^{d-1}: dz_0^{d-1}\right]&\\
            &=\left[-(z_1+1)^{d-1}: z_0^{d-1}\right].&
        \end{align*}
    Now writing $z_0 = z, z_1 = B(z)$ as in \eqref{eq_z_series} gives the local formula for $r$,
    \begin{align*}
        (y,z)&\mapsto \left(\frac{z^{d-1}}{y(1+B(z))}, z\right)&\\
        &=\left(\frac{z^{d-1}(1+O(z^d))}{y}, z\right).&
    \end{align*}
    \par\enspace(2) Next, we derive the formula for $s$. Since $s_{C,D}(p_0)=\{p_0\}$ by exceptionality (Lemma \ref{lem_e} (\ref{it_e_ind_r})) and $s_{C,D}$ is self-adjoint, $s_{C,D}$ defines a $(d-1)$-to-$(d-1)$ formal regular correspondence from $(C \times D,p_0)$ to itself, so $s:(\A^2,0)\vdash (\A^2,0)$ is a formal regular $(d-1)$-to-$(d-1)$ correspondence. 

    Since $J^1 \times J_*^1$ conjugates $s_{C,D}$ to $s_{J^1(C), J^1_*(D)}$ by Lemma \ref{lemma_COC}, we just need to express $s = s_{J^1(C), J^1_*(D)}$ in terms of $y,z$.
    
    Recall that $\Gamma_s = \Gamma_{s_+} - \Gamma_1.$ The local equations for $\Gamma_{s_+}$ in coordinates $y_0,y_1,z_0,z_1$ are
    \begin{align*}
        y_0 &= y'_0,\\
        y_1 &= y'_1,\\
        0 &= \begin{vmatrix}
            z_0 & z'_0 & y_0 \\
            z_1 & z'_1 &y_1\\
            1 &1&0\\
        \end{vmatrix}\\
        &=y_0(z'_1 - z_1) - y_1(z'_0 - z_0).
    \end{align*}

    In $yz$-coordinates, by \eqref{eq_z_series}, the equations for graph $\Gamma_{s_+}$ over the locus $y \neq 0$ are
    \begin{align*}
    y&=y',\\
    y(z'-z) &= \frac{1}{y}( B(z')-B(z)).
    \end{align*}
    Since $\Gamma_{s_+}$ is a dominant correspondence, it contains no component on which $y=0$ everywhere, so regular affine equations for the graph $\Gamma_{s_+}$ are
    \begin{align*}
    y&=y',\\
    y^2(z'-z) &= B(z')-B(z).
    \end{align*}
    We factor
    \begin{align*}
    y^2(z'-z) &= B(z')-B(z) \\
    &= -\frac{1}{d}(z'-z)(z^{d-1} + z^{d-2}z' + \ldots + (z')^{d-1}+ A(z,z'))
    \end{align*}
    where $A(z,z')$ is a formal power series containing only terms of degree at least $d$.
    
    The identity graph $\Gamma_1$ is defined by $y=y'$, $z=z'$, so $\Gamma_{s_+}-\Gamma_1$ is defined by 
    \[y=y',\]
    \[y^2 = -\frac{1}{d} \left( z^{d-1} + z^{d-2}z' + \ldots + (z')^{d-1}+ A(z,z') \right).\]
\end{proof}

\subsection{Constructing the model} \label{sect_P}

Now we introduce the model $P$ of $C \times D$, which we show in Proposition \ref{prop_essentially} is an essentially stable model for $b_{C,D}$. The idea is based on the connection between blowups and balls of Puiseux series; see \cite{MR2097722}. 

It is not hard to show that we can resolve destabilizing orbits terminating at $\Ind s$ with a single blowup. At each $p \in \Ind r$, we must work harder. Specifically, since $r$ is a family of rational maps over the parameter $z$, it acts on Puiseux series $y(z)$; the formula is 
$$r(y(z)) = \frac{z^{d-1}(1+O(z^d))}{y(z)} \quad \text{(Proposition \ref{prop_series})}.$$
So if $y(z)$ is a unit Puiseux series with leading term $u \in \bk \smallsetminus \{0\}$,
$$r(y(z)) = \frac{1}{u}z^{d-1} + \text{higher-order terms}.$$
This suggests adding divisors $E_k^{(p)}$ that parametrize Puiseux series of the form $y(z) = uz^k$, where $0< k \leq d-1$. On these series, $r$ acts by
$$r(y(z)) =  \frac{1}{u}z^{d-1 - k} + \text{higher-order terms}.$$ 
So we must perform a $(d-1)$-fold iterated blowup. There are many different $(d-1)$-fold iterated blowups, and the one we need is the simplest to define, but there is still quite a bit of notation required to define it in coordinates.

First, we define the blowup on an abstract $(\A^2,0)$.
\begin{defn}[standard iterated blowup]
Let $\tilde{\pi}_1: V_{\pi_1} \to (\A^2,0)$ be the simple blowup at \rev{$(0,0)$.} There is an open set $V_{\pi_1}^\circ \subset V_{\pi_1}$ with coordinates $(u_1,v_1)$ such that 
    $$\pi_1|_{V^\circ_{\pi_1}}(u_1,v_1)= (u_1v_1,v_1).$$ Let $E_1^\circ \subset V_{\pi_1}^\circ$ be the set $\{u_1 = 0\}$. Let $E_1$ be the Zariski closure in $V_{\pi_1}$ of $E_1^\circ$. Then \rev{$E_1 = \pi_1^{-1}(0,0)$} is the exceptional divisor of $\pi_1$. Now, let $(\A^2,0)_1$ be the germ about $(0,0)$ in $E_1^\circ$ in the $(u_1,v_1)$ coordinates. Repeating the above construction, we define $V_{\pi_2}$ and $V_{\pi_2}^\circ$, 
    \[\tilde{\pi}_2: V_{\pi_2} \to (\A^2,0)_1,\]
    where $\tilde{\pi}_2|_{V_{\pi_2}^\circ}: V_{\pi_2}^\circ \to (\A^2,0)_1$ is defined by $(u_2,v_2)\mapsto (u_2v_2, v_2)$. Iterating, we obtain a sequence of point blowups $\tilde{\pi}_1,\tilde{\pi}_2,\ldots$ . Let $\pi_k:V_{\pi_k}\to (\A^2,0)$ be 
    \[\pi_k \colonequals\tilde{\pi}_1\circ \ldots \circ \tilde{\pi}_k.\]
    By convention, $\pi_0$ is the identity. Note that the formula for $\pi_k$ on ${V^\circ_{\pi_k}}$ is $\pi_k (u_k, v_k) = (u_kv_k^k, v_k).$ Its exceptional divisor is 
    \[\pi_k^{-1}(0) = \sum_{j=1}^k E_j.\]
    We call $\pi_k$ the \emph{standard $k$-fold iterated blowup}.
\end{defn}

\begin{defn}[Construction of model $P$]\label{def_P}    
    In Proposition \ref{prop_series}, for each $p \in \Ind r$, we introduced preferred formal coordinates $(y,z)$ on $(C \times D, p) \cong (\A^2, 0)$. Denote the standard $k$-fold blowup in those coordinates by $\pi^{(p)}_k: V_{\pi^{(p)}_k} \to (C \times D, p)$. 

    Then $\pi: P\to C\times D$ is defined as the fiber product of $\pi^{(p)}_{d-1}$, as $p$ ranges over $\Ind r$, together with a simple blowup at each point of $\Ind s$. In other words, \rev{we blow up each indeterminacy point $r$ as described above and blow up each indeterminacy point of $s$ just once.}

Let $\mc{E}$ be the exceptional divisor of $\pi$. If $p\in \Ind s\cup \Ind r$, let $\mc{E}^{(p)}\colonequals \pi^{-1}(p)$. For each $p\in \Ind s$, let $E^{(p)}\colonequals \pi^{-1}(p)$. So 
\begin{align*}
    \mc{E} &= \sum_{p \,\in \,\Ind s \, \cup \, \Ind r} \mc{E}^{(p)} \\
    &= \sum_{p \, \in \, \Ind r}\sum_{k=1}^{d-1} E_k^{(p)} + \sum_{p \, \in \, \Ind s} E^{(p)}.
\end{align*}
Note that $\mc{E}$ has $2d^2$ components (see Figure \ref{fig_P}).

Define
\[  \hat{b},  \hat{r},  \hat{s} : P \vdash P, \]
\[ \hat{b} = \pi^{-1} \circ b \circ {\pi}, \quad \hat{r} = \pi^{-1} \circ r \circ {\pi}, \quad \hat{s} = \pi^{-1} \circ s \circ {\pi}.\]

\begin{figure}[h]
    \centering
    \includegraphics[scale=0.7]{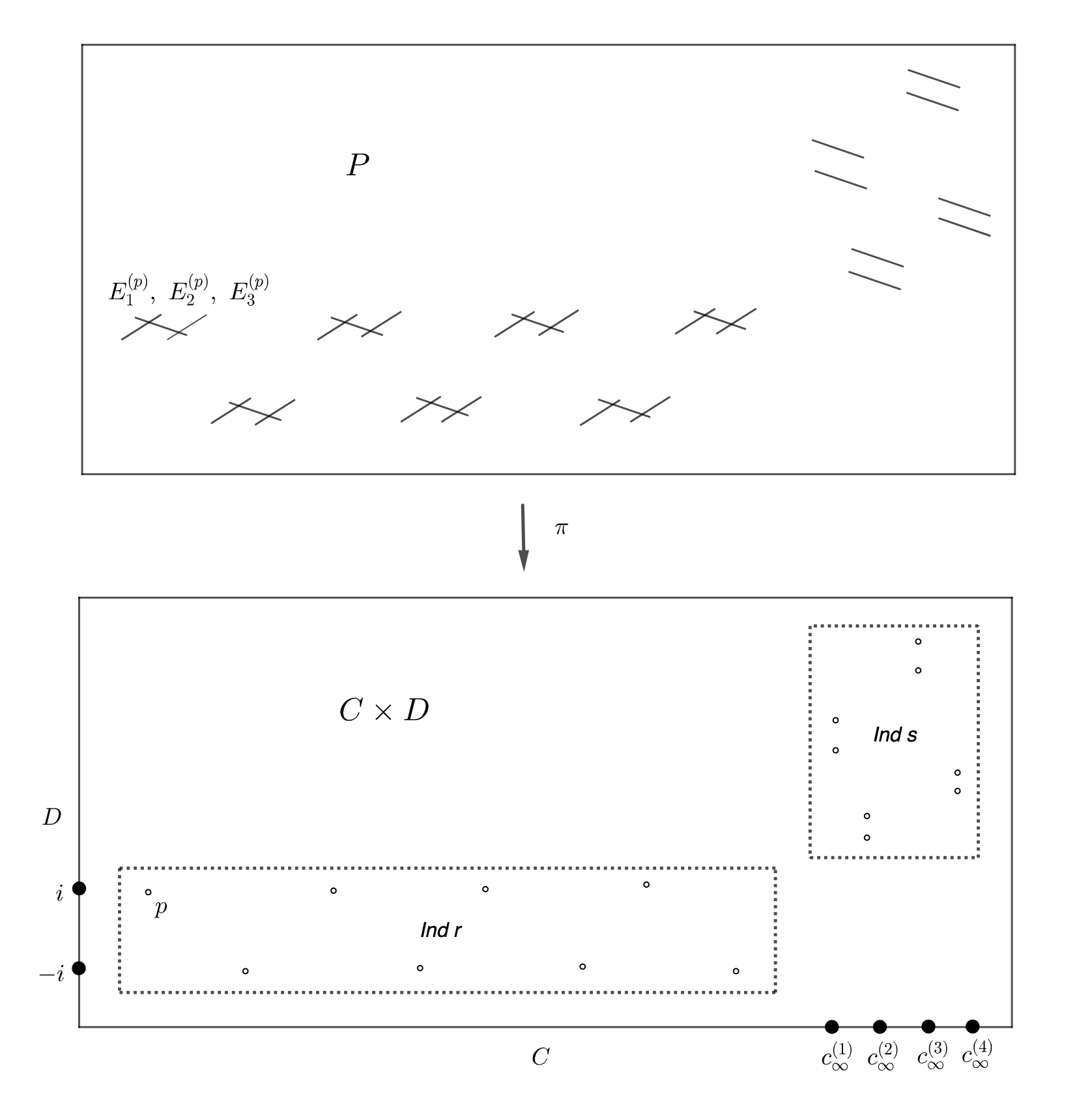}
    \caption{Model $P$, $d = 4$.}
    \label{fig_P}
\end{figure}

\begin{remark}
In \cite{billiardsI}, we introduced a modified phase space, i.e. a birational model, on which the degree growth of billiards in curves is easier to control. The modified phase space of the generic curve is the simple blowup of $C_{\gen}\times D$ at $\Ind s_{\gen} \cup \Ind r_{\gen}$. It is conceptually convenient to think of $P$, which is specific to the Fermat hyperbola, as a degeneration of $P_{\gen}$ in which some blowup centers have come together.
\end{remark}

\section{Reflection on model $P$} \label{sect_r}

Recall that $\pi: P \to C \times D$ is the model defined in Definition \ref{def_P}.
We show that the conjugate $\hat{r}: P \dashrightarrow P$ of $r$ is biregular.

For each $p = (x,v_\infty) \in \Ind r$, we defined $D^{(p)} = x \times D$.
By Lemma \ref{lem_r_basic}, the contractions of $r: C\times D \dashrightarrow C\times D$ are the curves 
\[D^{(p)} \times p.\]

Let $\hat{D}^{(p)} = (\pi^{-1})_{\str} (D^{(p)})$. To show that $\hat{r}$ has no contractions, we must check that the $\hat{D}^{(p)}$ and exceptional divisors $\mc{E}$ not are $\hat{r}$-contracted. 

The behavior of the exceptional primes above $\Ind s$ is the same as the general case of algebraic billiards.

\begin{lemma}[\protect{\cite[Lemma~6.5]{billiardsI}}]\label{lem_r_at_infty}
    Let $p \in \Ind s$, and let $E^{(p)}$ be the exceptional prime divisor above $p$. Then 
    \[\Ind \hat{r} \cap E^{(p)} = \varnothing,\]
    \[\hat{r}(E^{(p)}) = E^{(p)},\]
    \[\hat{r}_* E^{(p)} = E^{(p)}.\]
\end{lemma}

\begin{proof}
    Let $p = (x_\infty, v)$, where $x_\infty \in C_\infty$ and $[v] = t(x_{\infty})$. By Lemma \ref{lem_fh_geo} (\ref{it_non_isot}), $t(x_\infty)$ is non-isotropic. So $r_{x_\infty}(v) = v$, implying $r(p) = p$. The divisor $E^{(p)}$ is naturally the projectivized tangent space of $C \times D$ at $p$. Since $r(p) = p$ and $r$ is an involution, the differential of $r$ at $p$ is an involution, so $\hat{r}$ is regular on $E^{(p)}$ and defines an involution on $E^{(p)}$.
\end{proof}

It remains to check that the $\hat{D}^{(p)}$ and the $E^{(p)}_k$ above $p \in \Ind r$ are not $\hat{r}$-contracted curves.

\begin{lemma}\label{lem_r_A}
    Let $p \in \Ind r$, and let $E^{(p)}_0 \colonequals \hat{D}^{(p)}$. Then, for all $0 \leq k \leq d - 1$, we have
    \[\hat{r}_{\str}(E_k^{(p)})= E_{d-1-k}^{(p)}.\] 

\end{lemma}

\begin{proof}
    Fix $p$ and $k$. We claim that 
    \[(\pi_{d-1-k}^{-1}\circ r\circ \pi_k)_{\str}(E_k) = E_{d-1-k}.\]
    In coordinates $u = u_k, v = v_k$ on the source and $yz$-coordinates  (Proposition \ref{prop_series}) on the target (denoted here with primes $y', z'$ to clarify that these coordinates are on the target), the total composite graph of $r$ and $\pi_k$ has local formal equations 
    \begin{align*}
        v &=z',&\\
        y'v^ku &= (z')^{d-1} + O((z')^d)&\\
        &= v^{d-1} + O(v^d).&
    \end{align*}
    So $\Gamma_{r\circ \pi_k}$ has local formal equations 
    \begin{align*}
        v &=z',&\\
        y'u &= v^{d-1-k} + O(v^{d-k}).&
    \end{align*}
    So in coordinates $(u, v) = (u_k, v_k)$ on the source and $(u',v') = (u'_{d-1-k}, v'_{d-1-k})$ on the target, the total composite graph of $\pi_{d-1-k}^{-1}$ and $r\circ\pi_k$ has local equations 
    \begin{align*}
        v &=v',&\\
        u'(v')^{d-1-k}u &=v^{d-1-k}+O(v^{d-k})&\\
        &=(v')^{d-1-k} +O((v')^{d-k}).&
    \end{align*}
    So $\Gamma_{\pi_{d-1-k}^{-1}\circ r \circ \pi_k}$ maps $\{v=0\}$ to $\{v'=0\}$ via $u \mapsto 1/u$. 
    \rev{In other words, the following dominant restriction exists:
    \begin{equation} \label{eq_micro_r_preview}
        \hat{r}|_{E_k \vdash E_{d-1-k}} (u) = 1/u.
    \end{equation}
    }
\end{proof}

\begin{prop}\label{prop_r_reg}
    The birational map $\hat{r} : P \dashrightarrow P$ is in fact a biregular map $\hat{r}: P \to P$.
\end{prop}
\begin{proof}
     Since $\hat{r}=\hat{r}^{-1}$, we have $\Ind \hat{r}=\Exc \hat{r}$. Combining Lemma \ref{lem_r_at_infty} and Lemma \ref{lem_r_A}, we see that the strict transform of a curve by $\hat{r}$ is always $1$-dimensional. Since $\hat{r}$ is a $1$-to-$1$ rational correspondence, by Lemma \ref{lem_c_test} there are no $\hat{r}$-contracted curves, so $\Exc \hat{r}=\varnothing$.
\end{proof}

\section{Secant on model $P$} \label{sect_s}
In this section, we compute the images by $\hat{s}:P\vdash P$ of certain exceptional curves in $P$. The material in this section is limited to what is needed to compute the dynamical degree of $b_{C,D}$. More details on the behavior of $\hat{s}$ are given in Section \ref{sect_AS}, where we compute $\Ind \hat{s}$.

\rev{By Lemma \ref{lem_s_basic}, the contracted curves of $s: C\times D \vdash C\times D$ are the curves of the form 
$$C^{(p)} \colonequals C\times v,$$
where $p = (x_\infty, v) \in \Ind s$. The contractions are $C^{(p)} \times p \subset \Gamma_s$.}

\rev{Let $\hat{C}^{(p)}$ denote the strict transform of $C^{(p)}$ by the iterated blowup $\pi$. It is an irreducible curve in $P$. } The possible contracted curves of $\hat{s}$ are the $\hat{C}^{(p)}$ and the divisors in $\mc{E}$.

The behavior of the $\hat{C}^{(p)}$ and $E^{(p)}$ where $p \in \Ind s$ is the same as the generic case of algebraic billiards. We summarize the relevant facts in the next lemma.

\begin{lemma} \label{lem_s_infty}
    Let $p\in \Ind s$, and let $E^{(p)} \subset P$ be the exceptional prime \rev{divisor} above $p$. Then $\hat{C}^{(p)}$ and $E^{(p)}$ are not $\hat{s}$-contracted curves, and 
    \[\hat{s}_* E^{(p)} = \hat{s}_{\str}(E^{(p)}) = \hat{C}^{(p)}.\]
\end{lemma}

\begin{proof}
    By Lemma \ref{lem_s_basic}, we have $p = (x_\infty, v)$, where $x_\infty \in C_\infty$.  If $x \in C$ is general, the line $\ell(x, v)$ meets $C$ at $x$, $x_\infty$, and $d - 2$ other points, so $\hat{s}|_{\hat{C}^{(p)} \vdash \hat{C}^{(p)}}$ is $(d-2)$-to-$(d-2)$.  We showed that $C_\infty$ contains $d$ points, none with isotropic tangent slope (Lemma \ref{lem_fh_geo}). Then by \cite[Lemma 6.2]{billiardsI}, there is a natural identification of $E^{(p)}$ with the pencil of lines in $X$ through $x_\infty$, under which $\hat{s}|_{\hat{C}^{(p)} \vdash E^{(p)}}$ takes a general point $(x,v) \in C^{(p)}$ to the line through $x$ and $x_\infty$. It follows that $\hat{s}|_{\hat{C}^{(p)} \vdash E^{(p)}}$ is $(d-1)$-to-$1$.
    So, by the contracted curve test (Lemma \ref{lem_c_test}), the curve $\hat{C}^{(p)}$ is not $\hat{s}$-contracted.
    
    By symmetry, $\hat{s}|_{E^{(p)} \vdash \hat{C}^{(p)}}$ is $1$-to-$(d-1)$. This accounts for all images of points in $E^{(p)}$, and Lemma \ref{lem_c_test} shows that $E^{(p)}$ is not $\hat{s}$-contracted. Since $\Ind \hat{s} \cap \hat{C}^{(p)} = \varnothing$ \cite[Lemma 6.2]{billiardsI}, we obtain $\hat{s}_* E^{(p)} = \hat{s}_{\str}(E^{(p)})$.
\end{proof}

It follows that $s$ lifts to a biregular correspondence on the blowup of $C\times D$ at every point of $\Ind s$. The key question is how $s$ behaves above $\Ind r$. It turns out that the ``second half'' of the exceptional divisors above $\Ind r$ are $\hat{s}$-invariant.

\begin{lemma}\label{lem_s_lower}
    Let $p\in \Ind r$, and let $k$ be an integer in the range \rev{$\frac{d-1}{2} \leq k \leq d-1$.} Then 
    \[\hat{s}_{\str}(E^{(p)}_k) = (d-1)E^{(p)}_k.\]
    Further, the curve $E^{(p)}_k$ is not a $\hat{s}$-contracted curve.
\end{lemma}
\begin{proof}
    Fix $p$, and fix an integer $k$ in the range \rev{$\frac{d-1}{2} \leq k \leq d - 1$.} Let $E_k = E_k^{(p)}$. It suffices to show that $\deg_1(\hat{s}|_{E_k \vdash E_k}) = d-1$, since then by counting images there can be no more curves in $\hat{s}_{\str}(E_k)$. We claim that 
    \[(\pi_k^{-1} \circ s \circ \pi_k)_{\str}(E_k) = (d-1)E_k.\]
    In coordinates $(u,v) = (u_k, v_k)$ on the source and $y'z'$-coordinates on the target, the total composite graph of $s$ and $\pi_k$ has equations 
    \begin{align*}
        y' & = uv^k,&\\
        -d(uv^k)^2 &= v^{d-1}+v^{d-2}z' + \ldots + (z')^{d-1} + A(v,z'),&
    \end{align*}
    by Proposition \ref{prop_series} (\ref{item_prop_series_s}). The total composite graph has no non-dominant components, since these would lie above $\{v=0\}$ or $\{u=0\}$. So $\Gamma_{s\circ \pi_k}$ has the same equations. In coordinates $(u,v) = (u_k,v_k)$ on the source and $(u',v') = (u'_k, v'_k)$ on the target, the total composite graph of $\pi_k^{-1}$ and $s\circ \pi_k$ is then 
    \begin{align*}
        u'(v')^k&=uv^k,&\\
        -d(uv^k)^2 &= v^{d-1} + v^{d-2}v' + \ldots + (v')^{d-1} + A(v,v').&
    \end{align*}
    \rev{We proceed by calculating generic solutions $u', v'$ in terms of $u, v$, and then setting $v = 0$ to get the restriction to the exceptional divisor $E_k$. Let $\alpha = v'/v$. Since $A(v,v')\in \langle v,v'\rangle^d$, in coordinates $u,u',\alpha, v$, the graph is given by
        \begin{align*}
        u'v^k\alpha^{k}&= uv^{k},&\\
        -du^2v^{2k} & = v^{d-1}(\alpha^{d-1} + \ldots + 1 + O(v)).&
    \end{align*}
    Since $\frac{d-1}{2} \leq k$, we may cancel powers of $v$ to get the equations
    \begin{align*}
        u &= u'\alpha^{k},& \\
        -du^2 v^{2k - (d-1)} &= \alpha^{d-1} + \ldots + 1 + O(v).& 
    \end{align*}
    So given general values of $u, v$, for each of the $d-1$ choices of $\alpha$ with multiplicity, there is a value of $(u', v')$ corresponding to $(u,v)$. Setting $v = 0$, we get the following explicit formula for the dominant restriction,
    \begin{align}\label{eq_micro_s_preview}
        \hat{s}|_{E_k \vdash E_k} (u) = \left\{\frac{u}{\alpha^k} : \; \alpha\in \bk, \; \alpha^{d-1}+\ldots + 1 + w_{k} du^2 = 0\right\},
    \end{align}
    where
    \[
    w_{k} = \begin{cases} 1, & k = \frac{d-1}{2},\\ 0, & \textrm{otherwise}.
    \end{cases}
    \]
    }
    
    \rev{Finally, we note that by the contracted curve test (Lemma \ref{lem_c_test}), $E_k$ is not $\hat{s}$-contracted.}
\end{proof}

Now we can prove Theorem \ref{thm_main_model} (\ref{it_thm_d_3}).
\begin{cor} \label{cor_d_3_reg}
    If $d=3$, the correspondence $\hat{b}: P\vdash P$ is biregular. 
\end{cor}

\begin{proof}
    The correspondence $\hat{r}:P\to P$ is biregular (Proposition \ref{prop_r_reg}), so we just need to check that $\hat{s}$ is biregular. Recall that the possible contracted curves of $\hat{s}$ are the $\hat{C}^{(p)}$ and the divisors in $\mc{E}$. By Lemma \ref{lem_s_infty} and Lemma \ref{lem_s_lower} with $d = 3$, there are no contracted curves of $\hat{s}$, and $\hat{s}$ is self-adjoint. 
\end{proof}

If $d=2$, $P$ coincides with the modified phase space used to study general algebraic billiards, so $P$ is again a biregular model by \cite[Corollary 7.11]{billiardsI}.

\section{The dynamical degree} \label{sect_dd}

Usually, one calculates the dynamical degree of a correspondence by constructing an algebraically stable model. The issue in our case is that the construction of such a model depends in a delicate way on characteristic. Instead, we show that the model $P$ is essentially stable, Definition \ref{def_es}. That is, there exists a big-and-nef divisor $\Delta$ on $P$ such that, for all $n \in \N$,
    \begin{equation}
    \left(f^n\right)_*\Delta\cdot \Delta  = (f_*)^n \Delta \cdot \Delta. 
    \end{equation}
The idea for the proof is as follows. The lack of algebraic stability of a correspondence $f$ is measured by $(f^n)_* - (f_*)^n$, whose image is supported on $f$-exceptional image curves and their strict iterates. In the case of $\hat{b}$, these curves are all contained in a subset of $\mc{E}$. Then any big-and-nef divisor $\Delta$ orthogonal to every divisor in $\mc{E}$ suffices. Note that no ample divisor is orthogonal to $\mc{E}$, so we need to work with big-and-nef divisors.

Let $W_1\subset \Div P$ be the subgroup generated by $\{E^{(p)}_k : 1\leq k \leq d-2, \; p\in \Ind r\}$. 
Let $N_1$ be the image of $W_1$ in $\Num P$. 

\begin{lemma} \label{lem_W1}
\leavevmode
\begin{enumerate}
\item \label{it_W1_exc} All exceptional image curves of $\hat{s}, \hat{r}, \hat{b}$ lie in $W_1$.
\item \label{it_W1_str} We have
$$\hat{s}_{\str}(W_1) \subseteq W_1, \quad \hat{r}_{\str}(W_1) \subseteq W_1, \quad \hat{b}_{\str}(W_1) \subseteq W_1.$$
\item \label{it_W1_push} We have
$$\hat{s}_*(W_1) \subseteq W_1, \quad \hat{r}_*(W_1) \subseteq W_1, \quad \hat{b}_*(W_1) \subset W_1.$$
\item \label{it_N1} We have
$$\hat{s}_*(N_1) \subseteq N_1, \quad \hat{r}_*(N_1) \subseteq N_1, \quad \hat{b}_*(N_1) \subset N_1.$$
\end{enumerate}
\end{lemma}

\begin{proof}
    Claim (1) about $\hat{s}$ is by Lemma \ref{lem_s_lower}. Claim (1) about $\hat{r}$ is vacuous (there are no such curves), by Proposition \ref{prop_r_reg}.
    
    To see Claim (2) about $\hat{s}$, since each point of $\Ind r$ is $s$-exceptional (Lemma \ref{lem_e}), the strict transform by $\hat{s}$ of any curve in $W_1$ lies in $\pi^{-1}(\Ind r)$, but these images cannot land in any $E_{d-1}^{(p)}$, $p \in \Ind r$, by Lemma \ref{lem_s_lower}.

    Claim (2) about $\hat{r}$ is by Lemma \ref{lem_r_A}.

    Claims (1) and (2) about $\hat{b}$ follow from those for $\hat{s}$ and $\hat{r}$.

    Claim (3) follows from (1) and (2) because pushforward agrees with strict transform modulo exceptional image curves \rev{by Lemma \ref{lem_diff_support}}. 

    Claim (4) follows from (3) because pushforward respects the class map. 
\end{proof}

\begin{lemma}\label{lem_essentially_pre}
    For any divisor $\Delta\in \Div P$, and any $n\in \N$, we have 
    \[(\hat{b}^n)_* \Delta \equiv (\hat{b}_*)^n\Delta \mod W_1.\]
\end{lemma}
\begin{proof}
    We argue by induction. The case $n=1$ is trivial. Assuming the claim for $n$, we have
    \begin{align*}
        (\hat{b}^{n+1})_* \Delta &\equiv\hat{b}_* (\hat{b}^n)_* \Delta &  \text{(Lemma \ref{lem_W1} (\ref{it_W1_exc})}&\\
        &\equiv \hat{b}_*(\hat{b}_*)^n\Delta & \text{(Lemma \ref{lem_W1} (\ref{it_W1_push}))}&\\
        &\equiv (\hat{b}_*)^{n+1}\Delta.&
    \end{align*}
\end{proof}

Let $C_0$ and $D_0$ be divisors of the form $C\times q_0$ and $c_0 \times D$, where $c_0 \in C$, $q_0 \in D$ are general. Let $\hat{C}_0 = \pi^* C_0, \hat{D}_0 = \pi^* D_0$. Let $\Delta\colonequals \hat{C}_0 + \hat{D}_0$. 

\begin{prop}[essential stability] \label{prop_essentially}
The correspondence $\hat{b}:P \vdash P$ is essentially stable.
    It follows that
        $$\lambda_1(b) = \lambda_1(\hat{b}) = \rad\hat{b}_*.$$
\end{prop}
\begin{proof}
    The divisor $\pi_* \Delta$ is nef since it intersects $C_0$ and $D_0$ with multiplicity $1$ each, and these are a basis of $\Num C \times D$. So the pullback $\Delta = \pi^*\pi_* \Delta$ is nef. To show that $\Delta$ is also big, we need $\Delta \cdot \Delta > 0$. Since $\Delta$ is disjoint from the contracted curves of $\pi$, we have
    \[\Delta \cdot \Delta = \pi_*\Delta \cap \pi_*\Delta = 2 > 0.\] 
    
    By Lemma \ref{lem_essentially_pre}, for each $n \in \N$, we have 
    \[\left(\hat{b}^n\right)_* \Delta \in (\hat{b}_*)^n \Delta + W_1.\]
    Since $\Delta$ is intersection-orthogonal to $W_1$, intersecting against $\Delta$ proves that $\hat{b}$ is essentially stable.

    Dynamical degrees are birational conjugacy invariants (Theorem \ref{thm_dd_props} (\ref{it_birat_invariant})), so $\lambda_1(b)=\lambda_1(\hat{b})$. The remaining equality $\lambda_1(\hat{b}) = \rad\hat{b}_*$ is Lemma \ref{lem_AAS}.
\end{proof}

\rev{This completes the proof of Theorem \ref{thm_main_model} (\ref{it_thm_es}). To calculate the dynamical degree of the Fermat hyperbola billiard, it only remains to compute the spectral radius of $\hat{b}_*$.} Writing down a matrix for $\hat{b}_*$ would be a lot of work, since $\Num P$ has rank $2 + 2d^2$. Fortunately, there are some tricks that reduce the dimension of the computation. (Notice that the value we ultimately obtain in Theorem \ref{thm_main_dd} is at worst an algebraic integer of degree $2$ rather than $2 + 2d^2$.)

By Lemma \ref{lem_W1} (\ref{it_N1}), the pushforwards $\hat{s}_*, \hat{r}_*, \hat{b}_*$ descend to endomorphisms of $(\Num P)/ N_1$. Rather than study the whole space $(\Num P)/ N_1$, which has rank $2 + 4d$, it is more convenient to restrict our attention to just the divisors that show up when iterating $\hat{C}_0$ and $\hat{D}_0$. We set this up as follows.

Let $N$ be the smallest subspace of $\Num P$ that is closed under $\hat{s}_*$ and $\hat{r}_*$ and contains $\hat{C}_0$, $\hat{D}_0$, and $N_1$. Let $N_0 \colonequals N/N_1$. The pushforward homomorphisms $\hat{b}_*$, $\hat{r}_*$ $\hat{s}_*: \Num P \to \Num P$ map $N$ into $N$ and $N_1$ into $N_1$. Hence there are endomorphisms $\overline{b}$, $\overline{r}$ $\overline{s}: N_0 \to N_0$ induced by $\hat{b}_*|_N$, $\hat{r}_*|_N$, and $\hat{s}_*|_N$. Note that since $\hat{r}$ is regular, by Lemma \ref{lem_df}, we have $\hat{b}_* = \hat{r}_* \hat{s}_*$, so $\ovl{b} = \ovl{r} \, \ovl{s}$. 

Let
\[E_\infty \colonequals \sum_{p\in \Ind s} E^{(p)},\quad F_{d-1}\colonequals \sum_{p\in \Ind r} E^{(p)}_{d-1}.\]

\begin{prop}\label{prop_descend}
The group $N_0$ is a rank $4$ free abelian group generated by the classes of
\[\hat{C}_0, \hat{D}_0, E_\infty, F_{d-1}.\]
In this basis,
   \begin{equation}\label{eq_r_mat}
        \ovl{r} = \begin{bmatrix}
        1 & 0 & 0 & 0\\
        d(d-1) & 1 & 0 & 2d\\
        0 & 0 & 1 & 0 \\
        -(d-1) & 0 & 0 & -1
    \end{bmatrix}, \quad
        \ovl{s} = \begin{bmatrix}
            d-1 & 2 & d-1 & 0\\
            0 & d-1 & 0 & 0\\
            0 & -1 & -1 & 0\\
            0 & 0 & 0 & d-1
        \end{bmatrix},
  \end{equation}
  \begin{equation}
  \label{eq_b_mat} 
      \ovl{b} = \begin{bmatrix}
          d-1 & 2 & 2d & 0\\
          d^3-2d^2 +d & 2d^2-d-1  & 2d^3 & 2d^2-2d\\
          0 & -1 & -1 & 0\\
          -(d-1)^2 & -2d+2 & -(2d-2)d & -(d-1)
      \end{bmatrix}.
  \end{equation}
\end{prop}

To prove this, we show that, modulo $N_1$,
    \begin{align}
        \hat{r}_* \hat{C}_0 &\equiv \hat{C}_0 + d(d-1) \hat{D}_0 - (d-1)F_{d-1},\label{eq_descend_1}\\
        \hat{r}_* \hat{D}_0 &\equiv \hat{D}_0 ,\label{eq_descend_2}\\
        \hat{r}_* E_\infty &\equiv E_\infty,\label{eq_descend_3}\\
        \hat{r}_* F_{d-1} & \equiv 2d \hat{D}_0 - F_{d-1}, \label{eq_descend_4}\\
        \hat{s}_* \hat{C}_0 &\equiv (d-1) \hat{C}_0 ,\label{eq_descend_s1}\\
        \hat{s}_* \hat{D}_0 &\equiv 2 \hat{C}_0 + (d-1) \hat{D}_0 - E_\infty,\label{eq_descend_s2}\\
        \hat{s}_* E_\infty & \equiv (d-1)\hat{C}_0 - E_\infty,\label{eq_descend_s3}\\
        \hat{s}_* F_{d-1} &\equiv (d-1)F_{d-1}.\label{eq_descend_s4}
    \end{align}
    We prove \eqref{eq_descend_1}-\eqref{eq_descend_s4} in order.

    \begin{proof}[Proof of \eqref{eq_descend_1}]
    As in the case of a general algebraic billiard, $r_* C_0$ has bidegree $(1, d(d-1))$ in $\Num C \times D$ \cite[Equation 23]{billiardsI}. Therefore 
    $$\hat{r}_* \hat{C}_0 = \hat{C}_0 + d(d-1) \hat{D}_0 - \sum_{E \in \mc{E}} m_E E,$$
    where $m_E$ is the order of vanishing of $\pi^* r_* C_0$ along $E$.
Since $r_* C_0 \cap \Ind s = \varnothing$, we have $m_E = 0$ for all $E$ above $\Ind s$. Working modulo $N_1$, we just need the vanishing order of $r_* C_0$ on $E_{d-1}^{(p)}$, for each $p\in \Ind r$. By Proposition \ref{prop_series}, if $p\in \Ind r$, then locally around $p$ in $(y,z)$ coordinates the equation of $r_* C_0$ is of the form, for some $\kappa_0\neq 0$, 
    \[y = \kappa_0z^{d-1} + O(z^d).\]
    This vanishes to order $d-1$ on $E_{d-1}^{(p)}$, since setting $y= uv^{d-1}$, $z= v$, the equation is sharply divided by $v^{d-1}$. 
    \end{proof}

    \begin{proof}[Proof of \eqref{eq_descend_2}]
    Indeed $\hat{r}_* \hat{D}_0 = \hat{D}_0$ as divisors, since $r_* D_0 = D_0$ and $D_0 \cap \Ind \pi=\varnothing$. 
    \end{proof}

    \begin{proof}[Proof of \eqref{eq_descend_3}]
    Indeed $\hat{r}_* E_\infty = E_\infty$ as divisors (Lemma \ref{lem_r_at_infty}).
    \end{proof}

    \begin{proof}[Proof of \eqref{eq_descend_4}]
    In fact, we claim that for each $p\in \Ind r$, we have 
    \[\hat{r}_*(E_{d-1}^{(p)}) \equiv D_0 - E_{d-1}^{(p)}\mod N_1.\]
    By Lemma \ref{lem_r_A}, we have
    \[ (r \circ \pi )_* E_{d-1}^{(p)} = D^{(p)},\]
    which is of bidegree $(0,1)$. So
    $$\hat{r}_* E_{d-1}^{(p)} = \hat{D}_0 - \sum_{E \in \mc{E}} m_E E,$$
    where $m_E$ is the order of vanishing of $\pi^* D^{(p)}$ along $E$. Since $D^{(p)} \cap (\Ind s \cup \Ind r) = \{ p \}$, working modulo $N_1$, we just need to compute $m_{E_{d-1}^{(p)}}$. Writing $D^{(p)}$ locally at $p$ as $\{z=0\}$, $\pi^* D^{(p)}$ has local equation $\{v_{d-1} =0\}$ in $V^\circ_{\pi_{d-1}}$, same as $E^{(p)}_{d-1}$.
    \end{proof}

    \begin{proof}[Proof of \eqref{eq_descend_s1}]
    Since $q_0$ is general, $s|_{C_0 \vdash C_0}$ is $(d-1)$-to-$(d-1)$, so $s_* C_0 = (d-1) C_0$.
    \end{proof}

    \begin{proof}[Proof of \eqref{eq_descend_s2}]
    As in the case of a general algebraic billiard,
    $s_* D_0$ has bidegree $(d-1, 2)$ in $\Num C \times D$ \cite[Equation 22]{billiardsI}. Therefore 
    $$\hat{s}_* \hat{D}_0 = (d-1)\hat{C}_0 + 2 \hat{D}_0 - \sum_{E \in \mc{E}} m_E E,$$
    where $m_E$ is the order of vanishing of $\pi^* s_* D_0$ along $E$. 

    Similarly, for each $p \in \Ind s$, we have $\hat{s}_* E^{(p)} = C_0 - E^{(p)}$ in $\Num P$, since the divisor $(s \circ \pi)_* E^{(p)}$ is a smooth curve of bidegree $(1,0)$ such that $(s \circ \pi)_* E^{(p)} \cap (\Ind s \cup \Ind r) = \{p\} $ (Lemma \ref{lem_s_infty}).
    
    Notice $D' \cap (\Ind s \cup \Ind r) = \Ind s$. So in $\Num P$,
    \begin{align*}
        m_{E^{(p)}} = \hat{s}_* D_0 \cdot E^{(p)} & = D_0
        \cdot \hat{s}_* E^{(p)} & \text{(by self-adjointness)}\\
        & = D_0 \cdot(C_0-E^{(p)}) \\
        & = 1.&
    \end{align*}
    \end{proof}
    
    \begin{proof}[Proof of \eqref{eq_descend_s3}]
    Again $\hat{s}_* E^{(p)} = C_0 - E^{(p)}$ in $\Num P$ by Lemma \ref{lem_s_infty} and linear equivalence. 
    \end{proof}

    \begin{proof}[Proof of \eqref{eq_descend_s4}]
    We have $\hat{s}_{\str} F_{d-1} = (d-1)F_{d-1}$ by Lemma \ref{lem_s_lower}. And $\hat{s}_* F_{d-1} - \hat{s}_{\str} F_{d-1}$ is supported on $\hat{s}_*$-exceptional image curves, which are in $W_1$ by Lemma \ref{lem_W1} (\ref{it_W1_exc}). 
\end{proof}

We arrive at the proof of our main result on the Fermat hyperbola.

\begin{thm} \label{thm_dd_body}
    The dynamical degree of the billiards correspondence $b: C\times D \vdash C\times D $ in the Fermat hyperbola is
    \begin{equation} \label{eq_dd_body}
        \lambda_1(b) = \frac{2d^2 - 3d + \sqrt{(2d^2 - 3d)^2 - 4(d-1)}}{2}.
    \end{equation}
\end{thm}
\begin{proof}
    By Proposition \ref{prop_essentially}, 
    $$\lambda_1(b) = \lambda_1(\hat{b}) = \rad\hat{b}_*.$$
    We claim $\rad\hat{b}_* = \rad\ovl{b}$.
Since $\hat{b}_*|_{N_0}$ is semiconjugate to $\ovl{b}$, see Proposition \ref{prop_descend}, we have 
    \[\rad \hat{b}_* \geq \rad \hat{b}|_{N_0} \geq \rad \ovl{b}.\]
    To prove the reverse inequality, we use an ad hoc argument. For all $n$, $(\hat{b}_*)^n \Delta$ is in \rev{$N$.} Writing the image of $(\hat{b}_*)^n$ in \rev{$N_0$} in our preferred basis, and using that $\Delta$ is orthogonal to $N_1$, we have
    \begin{align*}
        (\hat{b}_*)^n \Delta\cdot \Delta &= \tiny \begin{bmatrix}
            1 & 1 & 0 & 0
        \end{bmatrix} \ovl{b}^n \tiny \begin{bmatrix}
            1\\1\\0\\0
        \end{bmatrix}.
    \end{align*}
    \rev{So for each $\varepsilon> 0$, there is a constant $c_\varepsilon > 0$ such that for all $n \geq 0$, we have
    $$(\hat{b}_*)^n \Delta\cdot \Delta \leq c_\varepsilon (\rad\ovl{b})^{n+\varepsilon}).$$
    }
    In the limit as $n \to \infty$,
    \begin{align*}
        \rad \hat{b}_* &= \lim_{n\to \infty} \left((\hat{b}_*)^n \Delta\cdot \Delta\right)^{1/n} &\text{(Proof of Lemma \ref{lem_AAS})}\\
        & \leq \lim_{n \to \infty} \left( \rev{ c_\varepsilon} (\rad\ovl{b} )^{n+\varepsilon} \right)^{1/n} \\
        &  \leq \rad \ovl{b} + \varepsilon.
    \end{align*}
    Since this holds for all $\epsilon > 0$, we have $\rad\hat{b}_* \leq \rad\ovl{b}$.

    Finally, we compute $\rad \ovl{b}$. By \eqref{eq_b_mat}, the characteristic polynomial of $\ovl{b}$ is
\[\det(\lambda I - \ovl{b}) = (\lambda - (d-1))^2(\lambda^2 - (2d^3 - 2d)\lambda + d-1).\]
\rev{Since $d\geq 2$, the quadratic factor is real-rooted, and its maximal root is equal to the right side of \eqref{eq_dd_body}. The latter root is at least $d-1$, so it is the spectral radius.}
\end{proof}

Finally, we specialize the generic curve to obtain our main theorem.

\begin{proof}[Proof of Theorem  \ref{thm_main_dd}]
    Let $C_{\gen}$ be the generic curve of degree $d\geq 2$, and let $D$ be defined by $Q_0^2 + Q_1^2 = Q^2_2$. By Theorem \ref{thm_dd_props} (\ref{it_specialization}), the dynamical degree of $b_{{\gen}}: C_{\gen} \times D \vdash C_{\gen} \times D$, viewed as a correspondence over the geometric function field of the parameter space of degree $d$ curves, is at least $\lambda_1(b)$, which we computed in Theorem \ref{thm_dd_body}. This proves the theorem for the Euclidean form. But the claim is independent of the choice of form, since all nondegenerate quadratic forms are equivalent up to change of coordinates, and coordinate changes preserve the generic curve. 
\end{proof}

Theorem \ref{thm_main_dd} implies that the lower bound on $\lambda_1$ holds for the billiard correspondence in a very general curve, where \emph{very general} requires that the coefficients appearing in an affine equation of the curve are algebraically independent.

\section{The midpoint divisor, algebraic stability, and periodic orbits} \label{sect_AS}

We have shown that $(P, \hat{b})$ is an essentially stable model for the Fermat hyperbola billiard. By a more careful study of the exceptional divisors, we can describe the indeterminacy of $\hat{b}$ completely. If $d$ is odd, then over $\bk = \C$, we show that $P$ is in fact an algebraically stable model (Theorem \ref{thm_main_model} (\ref{it_thm_as})). A minor modification of $P$ is needed if $d$ is even. For the sake of clarity, we present the odd case first and leave the even case for Section \ref{sect_even}.

\subsection{Odd case}

Recall the structure of the exceptional divisor of $P$,
\begin{align*}
    \mc{E} &= \sum_{p \,\in \,\Ind s \, \cup \, \Ind r} \mc{E}^{(p)} \\
    &= \sum_{p \, \in \, \Ind r}\sum_{k=1}^{d-1} E_k^{(p)} + \sum_{p \, \in \, \Ind s} E^{(p)}.
\end{align*}

If $d$ is odd, the \emph{midpoint divisor} above $p \in \Ind r$ is $E^{(p)}_{(d-1)/2}$. It plays a special role in the billiard dynamics. We set the notation
$$E_{\Mid}^{(p)} \colonequals E^{(p)}_{(d-1)/2}.$$
If $p$ is understood, we just write $E_{\Mid}$.
\end{defn}

If we let $E_0^{(p)}= \{z=0\}$, then $E_{(d-1)/2}^{(p)}$ is the midpoint of the dual graph of $\{E_0^{(p)}, \ldots, E_{d-1}^{(p)}\}.$

We identify $E_{\Mid}$ with $\PP^1$ by the map $u = u_{(d-1)/2}$, where $u_{(d-1)/2}: E_{\Mid}^\circ \to \A^1$, see Definition \ref{def_P}. The point $a = E_{\Mid}\cap E_{(d-1)/2-1}$ is thus identified with $\infty$ in $\PP^1$. 

\begin{lemma}\label{lem_micro_s}
The following dominant restrictions exist:
\[\rho\colonequals \hat{r}|_{E_{\Mid}\vdash E_{\Mid}}, \quad \sigma \colonequals \hat{s}|_{E_{\Mid}\vdash E_{\Mid}}.\]
The map $\rho$ is the involution
\begin{align}\label{eq_micro_r}
    \rho(u) &= \frac{1}{u}.
\end{align} 
    The correspondence $\sigma$ is $(d-1)$-to-$(d-1)$, acting on $\A^1$ by
    \begin{align}\label{eq_micro_s}
        \sigma(u) = \left\{\frac{u}{\alpha^{(d-1)/2}} : \; \alpha\in \bk, \; \alpha^{d-1}+\ldots + 1 + \rev{du^2} = 0\right\}.
    \end{align}
    The image of $a = \infty$ is $\sigma(\infty) = \{ u \in \bk : u^2 = -1/d\}$, each with multiplicity $(d-1)/2$. 
\end{lemma}

\begin{proof}
\rev{Equation \eqref{eq_micro_r} and \eqref{eq_micro_s} are the special cases of Equation \eqref{eq_micro_r_preview} and Equation \eqref{eq_micro_s_preview} with $k = (d-1)/2$.}

By self-adjointness of $\hat{s}$, $\sigma(\infty) = \sigma^{\adj}(\infty)$. If $u_0\in \bk$, and $\sigma(u_0)$ is finite, then $0$ is a root of $\alpha^{d-1}+\ldots + 1 + du_0^2$ as a polynomial in $\alpha$, so $u_0^2 = -1/d$. So it remains to show that $\infty$ is not in $\sigma(\infty)$. By \eqref{eq_micro_s}, we have $0 = 1 + \ldots + 1/\alpha^{d-1}+ d(u')^2$ for any $u'$ finite in $\sigma(u)$. As $u\to \infty$, we get $\alpha\to \infty$, so $(u')^2 \to -1/d$. The multiplicities follow by Galois symmetry. 
\end{proof}

We may now describe the indeterminacy points of $\hat{s}$.
\begin{lemma}\label{lem_ind_s}
    Say $d$ is odd. The $\hat{s}$-contracted curves and $\hat{s}$-exceptional image curves are the divisors $E^{(p)}_k$, $p\in \Ind r$, $1\leq k < (d-1)/2$. For each $p\in \Ind r$, 
    \[\Ind \hat{s} \cap \mc{E}^{(p)} \subset E^{(p)}_{\Mid},\]
    and with the identification $u: E^{(p)}_{\Mid}\to \PP^1$, 
    \begin{equation} \label{eq_ind_s}
    \Ind\hat{s} \cap \mc{E}^{(p)} = \sigma(\infty) = \{u \in \bk : u^2 = -1/d \}.
    \end{equation}
    For each $u_0\in \Ind\hat{s}\cap \mc{E}^{(p)}$, we have \[\hat{s}(u_0) = \bigcup_{1\leq k < (d-1)/2} E^{(p)}_k.\]
\end{lemma}

\begin{proof}
    Fix $p \in \Ind r$. We argue by descending induction on $k$, from $k = (d-1)/2 - 1$ to $1$. Since $s(p) = \{p\}$, we have $\hat{s}(E_k^{(p)}) \subset \mc{E}^{(p)}$. Let $k = (d-1)/2 - 1$, $a^{(p)} \colonequals E^{(p)}_{(d-1)/2-1}\cap E^{(p)}_{\Mid}$. So $u(a^{(p)}) = \infty$. We claim we cannot have $E^{(p)}_{\Mid}\subset \hat{s}(E^{(p)}_{(d-1)/2-1})$. By \rev{Equation \eqref{eq_micro_s} and self-adjointness of $\hat{s}$,} a general point of $E_{\Mid}$ has a discrete set of at least $d-1$ preimages in $E_{\Mid}$ with multiplicity, none of them $\infty$. Having a further preimage is not possible by Stein (Lemma \ref{lem_stein}). Thus $\hat{s}(E^{(p)}_{(d-1)/2 -1})$ contains $\sigma(\infty)$ as a discrete set, with total multiplicity $d-1$. It follows, again by Lemma \ref{lem_stein}, that $\hat{s}(E_{(d-1)/2-1})$ contains no other points. So $\hat{s}(E^{(p)}_{(d-1)/2 -1}) = \sigma(\infty)$. Repeating this argument for $k-1$, we find that $\hat{s}(E^{(p)}_{k-1})\subset \hat{s} (E^{(p)}_{k-1} \cap E^{(p)}_k) = \Ind \hat{s}$ with images of total multiplicity $d-1$, etc. 
\end{proof}
\rev{It is then easy to describe the exceptional points and indeterminacy points of $\hat{b}$.}
\begin{prop} \label{prob_b_exc_ind}
    If $d$ is odd,
    \rev{
    \[\Ind \hat{b} \subset \bigcup_{p \in \Ind r} E^{(p)}_{\Mid},\]
    \[\Exc \hat{b} \subset \bigcup_{p \in \Ind r} E^{(p)}_{\Mid}.\]
    }
    Given $p \in \Ind r$, under the identification $u: E^{(p)}_{\Mid}\to \PP^1$, 
    \rev{
    \[\Ind\hat{b} \cap \mc{E}^{(p)} = \{u \in \bk : u^2 = -1/d \},\]
    \[\Exc\hat{b} \cap \mc{E}^{(p)} = \{ u \in \bk : u^2 = -d\}.\]
    }
    Further, the dominant restriction
    \[\beta \colonequals \hat{b}|_{E_{\Mid}\vdash E_{\Mid}}
    \] 
    is $(d-1)$-to-$(d-1)$.
\end{prop}
\rev{
\begin{proof}
Since $\hat{r}$ is biregular (Proposition \ref{prop_r_reg}), we have $\Ind \hat{b} = \Ind \hat{s}$, and the latter set is described by \eqref{eq_ind_s} in Lemma \ref{lem_ind_s}. By biregularity of $\hat{r}$ and self-adjointness of $\hat{s}$, we have 
$$\Exc \hat{b} = \Exc (\hat{r} \circ \hat{s}) = \hat{r}(\Exc \hat{s}) = \hat{r}(\Ind \hat{s}),$$
and $\hat{r}$ acts by $u \mapsto 1/u$.
\end{proof}
}

The midpoint divisor contains all the important points for the analysis of algebraic stability. Further, it is forward-invariant and backward-invariant for $\hat{b}_{\str}$. This reduces the problem to the dynamics of $\beta$. While a $1$-dimensional correspondence may have very complicated dynamics, this one turns out to be easy to understand over $\C$ due to the existence of an invariant subset. The idea is communicated in Figure \ref{fig_tower_odd}.

\begin{lemma} \label{lem_invt_o}
Let $\bk = \C$, assume $d$ is odd, and say $p \in \Ind r$. Let 
$$U = \C \smallsetminus \D = \{u: 1 \leq \abs{u} < \infty \} \subset E_{\Mid}^{(p)}.$$ Then $\hat{b}(U) \subset U$.
\end{lemma}

\begin{proof}
Since $\Ind \hat{b} \cap U = \varnothing$, if $u \in U$, then $\hat{b}(u)$ is a set of points; then since $\beta$ is $(d-1)$-to-$(d-1)$, $\hat{b}(u) = \beta(u)$. So it suffices to show $\beta(U) \subset U$.

Indeed, suppose by way of contradiction that $u \in U$, $u'\in \beta(u)$, and $\abs{u'} < 1$. By counting degrees, we have $\beta = \rho \circ \sigma$. By \eqref{eq_micro_r} and \eqref{eq_micro_s}, we have 
\[\beta(u) = \left\{\frac{\alpha^{(d-1)/2}}{u} : \; \alpha\in \C, \; \alpha^{d-1} + \ldots + 1 + \rev{du^2} =0\right\}.\]
Rearranging, there exists $\alpha \in \C$ satisfying
$$\abs{u'} = \abs{\frac{\alpha^{(d-1)/2}}{\rev{u}} } < 1,$$
$$ \frac{\alpha^{d-1}}{u^2} + \ldots + \frac{1}{u^2} = -d.$$
Each term in the sum is a product of powers of $1/u$ and $\alpha^{(d-1)/2}/{u}$. Since $\abs{1/u} \leq 1$ and $\abs{\alpha^{(d-1)/2}/{u}}<1$, the triangle inequality gives $d < \abs{-d}$, a contradiction.
\end{proof}

\begin{figure}[b]
    \centering \includegraphics[width=0.375\linewidth]{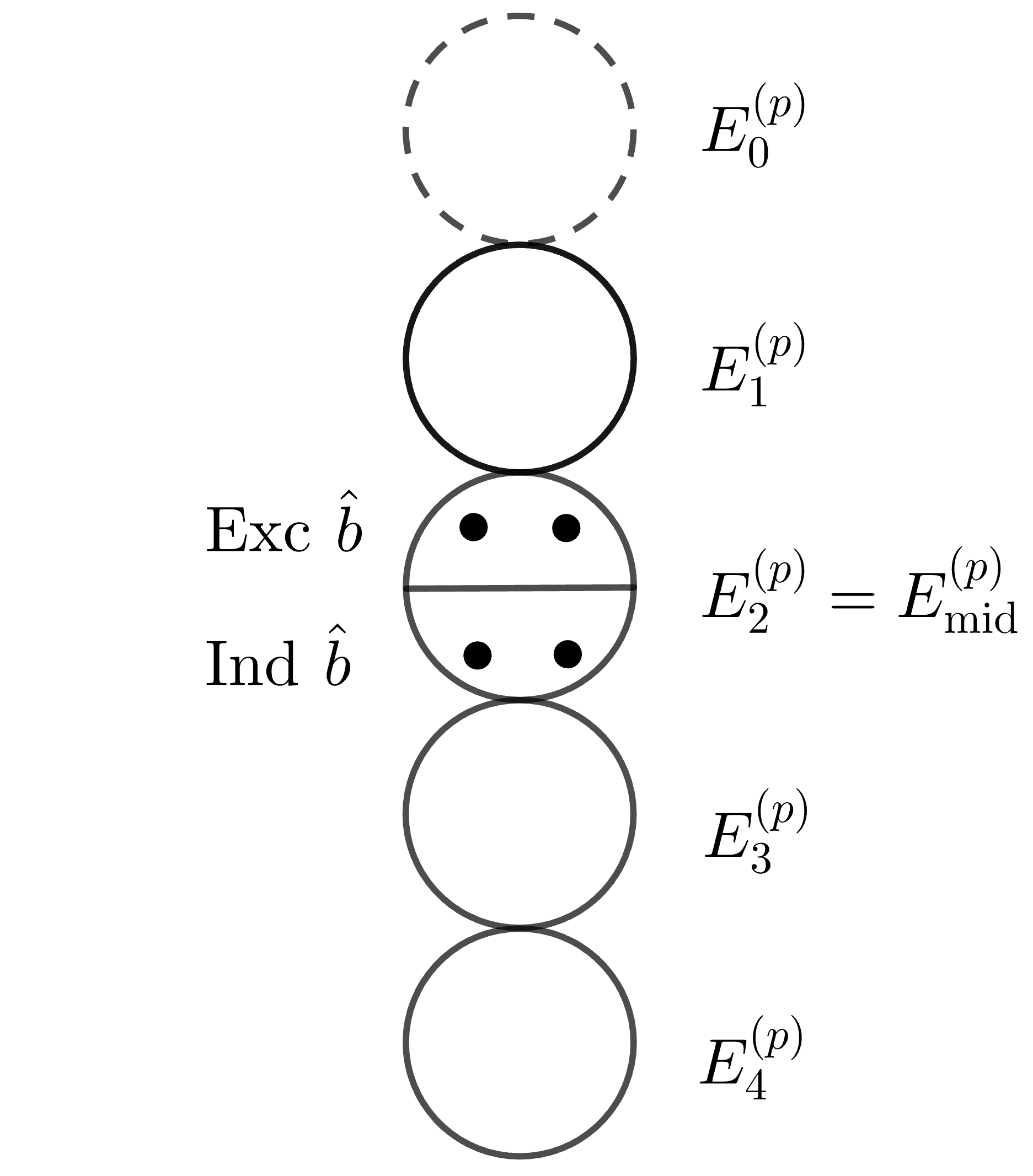}
    \caption{Dynamics on $\mc{E}^{(p)}$, $p \in \Ind r$, when $d = 5$. Read the figure as a chain of five Riemann spheres viewed in profile. The top sphere is $\hat{D}^{(p)} = E_0^{(p)}$. The bottom four spheres form $\mc{E}^{(p)}$. The reflection map $\hat{r}$ rotates the diagram, exchanging the top and bottom sphere. The horizontal line depicts the unit circle on $E^{(p)}_{\Mid}$. The two dots above the unit circle are $\Exc \hat{b} \cap \mc{E}^{(p)}$; the only contracted curve is $E_1^{(p)}$. The two dots below the unit circle are $\Ind \hat{b} \cap \mc{E}^{(p)}$. The map $\hat{b}$ sends the upper half of $E^{(p)}_{\Mid}$ into itself, so the forward orbit of $\Exc \hat{b}$ never meets $\Ind \hat{b}$.}
    \label{fig_tower_odd}
\end{figure}

Now we prove the odd case of Theorem \ref{thm_main_model} (\ref{it_thm_as}).

\begin{thm} \label{thm_AS_odd}
    Let $\bk = \C$. If the degree $d$ is odd, then $\hat{b}: P\vdash P$ is algebraically stable.
\end{thm}

\begin{proof}
By Theorem \ref{thm_dd_props} (\ref{it_as}), the correspondence $\hat{b}:P\vdash P$ is algebraically stable if for all $n\geq 0$,
\[\hat{b}^n(\Exc \hat{b}) \cap \Ind \hat{b} = \varnothing.\]
Let $e_0 \in \Exc \hat{b}$. By Proposition \ref{prob_b_exc_ind}, $e_0 \in E_{\Mid}^{(p)} \cong \hat{\C}$ for some $p \in \Ind r$. Let $U = \C \smallsetminus \D \subset E_{\Mid}^{(p)}$. By Proposition \ref{prob_b_exc_ind}, $e_0 \in U$ and $\Ind \hat{b} \cap U = \varnothing$. Since $\hat{b}(U) \subset U$ by Lemma \ref{lem_invt_o}, we have $\hat{b}^n(\Exc \hat{b}) \subset U$ for all $n \geq 0$, so $\hat{b}^n(\Exc \hat{b}) \cap \Ind \hat{b} = \varnothing$.
\end{proof}

\subsection{Even degree case of algebraic stability} \label{sect_even}

To study the case of even degree $d$, we will require a model that has an analogue of the midpoint divisor. We therefore introduce the following variant of $P$.

\begin{defn}[Model $P_+$] \label{def_pplus}
    Suppose $d$ is even. Let $p\in \Ind r$. Let $a^{(p)} = E_{d/2 -1}^{(p)} \cap E_{d/2}^{(p)}$. Let $\tilde{\pi}^{(p)}_+:P_+^{(p)} \to P$ be the simple point blowup of $P$ at $a^{(p)}$. Let $\tilde{\pi}_+: P_+\to P$ be the join of all the $\tilde{\pi}_+^{(p)}$, and let $\pi_+ = \pi\circ \tilde{\pi}_+$. Let $\mc{E}_+$ be the set of irreducible components of $\pi_+^{-1}(\Ind r \cup \Ind s)$, and let $\mc{E}^{(p)}_+$ be the subset of $\mc{E}_+$ over a given point $p$. See Figure \ref{fig_P_plus}.

    Given $p \in \Ind r$, the \emph{midpoint divisor} above $p$ is
    $$E_{\Mid}^{(p)} \colonequals \tilde{\pi}_+^{-1}(a^{(p)}).$$
    If $p$ is understood, we just write $E_{\Mid}$.

    Let 
    $$\hat{b}_+, \hat{r}_+, \hat{s}_+ \colon P_+ \vdash P_+,$$
    $$\hat{b}_+ \colonequals \pi_+^{-1} \circ b \circ \pi_+, \quad \hat{r} \colonequals \pi_+^{-1} \circ r \circ \pi_+, \quad \hat{s} \colonequals \pi_+^{-1} \circ s \circ \pi_+.$$

Given $p \in \Ind r$, we let
\[\beta_+ \colonequals \hat{b}_+|_{E_{\Mid}\vdash E_{\Mid}}, \quad \rho_+ \colonequals \hat{r}_+|_{E_{\Mid}\vdash E_{\Mid}}, \quad \sigma_+ \colonequals \hat{s}_+|_{E_{\Mid}\vdash E_{\Mid}}.\]    
\end{defn}
    
Note that $P_+$ requires a satellite blowup, unlike $P$. From the Favre-Jonsson point of view \cite[Chapter 6]{MR2097722}, the divisor $E_{\Mid}^{(p)}$ is the vertex of the relative universal dual graph that parametrizes Puiseux series of the form $y(z) = uz^{(d-1)/2}$.

\begin{figure}[h]
    \centering
    \includegraphics[scale=0.7]{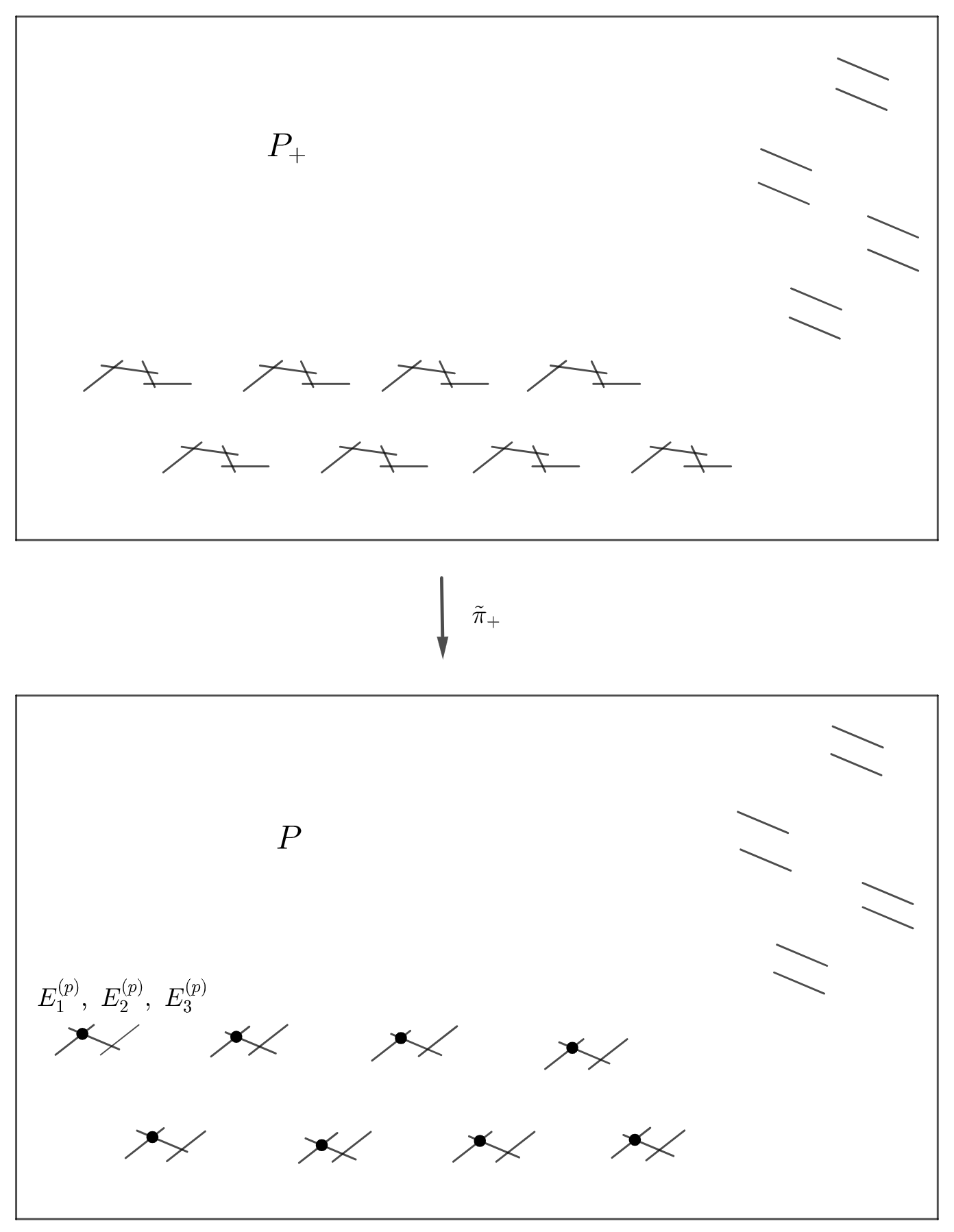}
    \caption{Model $P_+$, $d = 4$. The dots in $P$ represent blowup centers.}
    \label{fig_P_plus}
\end{figure}

\begin{figure}[b]
    \centering \includegraphics[width=0.24\linewidth]{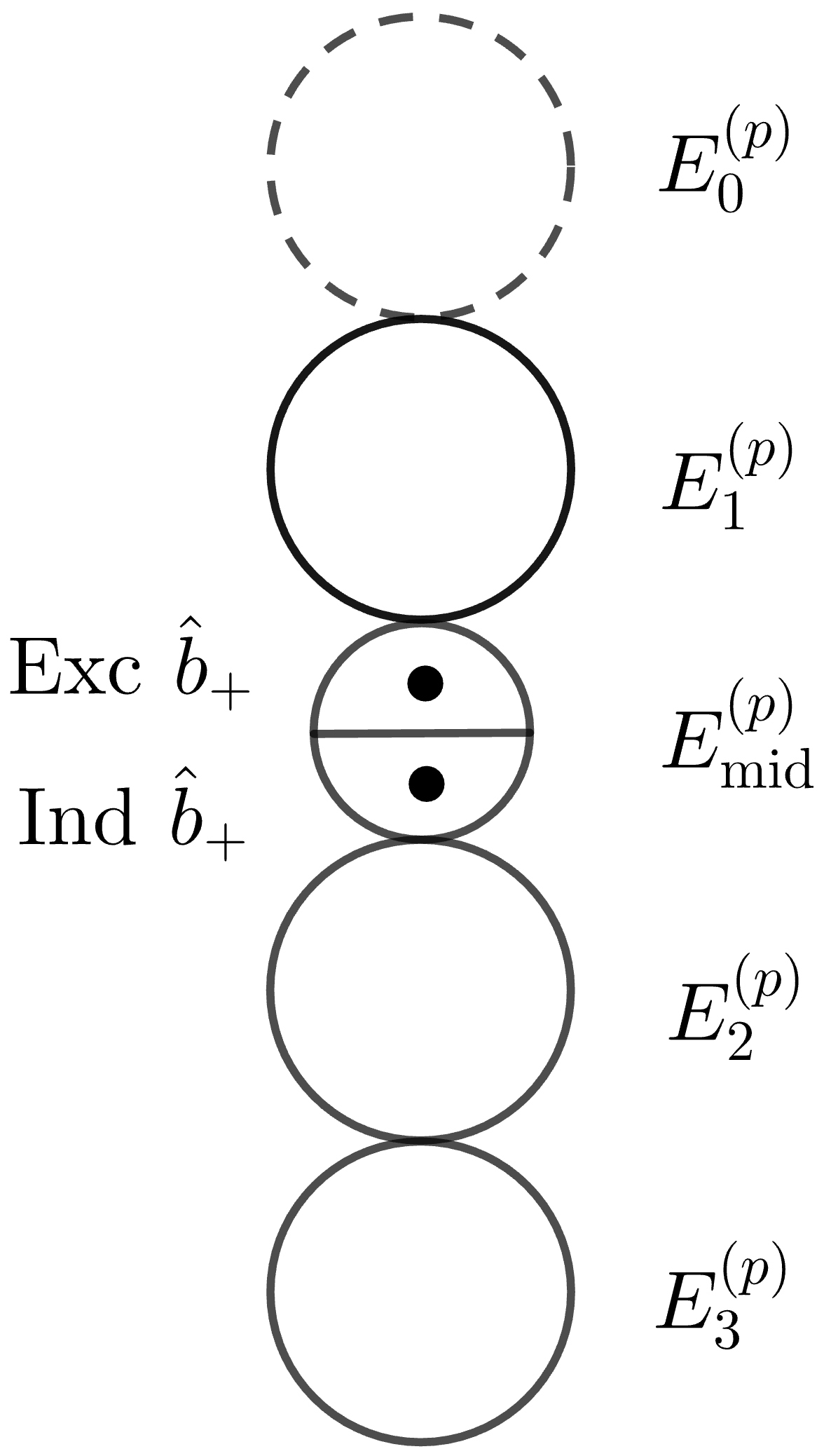}
    \caption{Dynamics on $\mc{E}_+^{(p)}$, $p \in \Ind r$, when $d = 4$. Read the figure as in Figure \ref{fig_tower_odd}. The reflection map $\hat{r}_+$ again rotates the diagram. The horizontal line depicts the unit circle on $E^{(p)}_{\Mid}$. The dot above the unit circle is the sole  point of $\Ind \hat{b}_+ \cap \mc{E}^{(p)}$. The dot below the unit circle is the sole point in $\Exc \hat{b}_+ \cap \mc{E}^{(p)}$. The only contracted curve is $E_1^{(p)}$. The dynamics are analogous.}
    \label{fig_tower_even}
\end{figure}

In order to avoid using Puiseux series, when $d$ is even, we study $P_+$ using a local double cover. Let $p \in \Ind r$. Let $G$ be the formal neighborhood of $C \times D$ about $p$ with the standard coordinates $y,z$, defined in \eqref{eq_yz}. We construct a double cover of $G$ as follows. 

Let $\tilde{G} = \tilde{G}^{(p)} = (\A^2,0)$ with coordinates $(\tilde{y},\tilde{z})$, equipped with a $2$-to-$1$ map 
$$\delta: \tilde{G} \to G,$$
$$(\tilde{y},\tilde{z}) \mapsto (\tilde{y}, \tilde{z}^2).$$

Let $\tau: \tilde{P}^{(p)} \to \tilde{G}$ denote the standard $(2d-2)$-fold iterated blowup defined using the process in Definition \ref{def_P}, with respect to $\tilde{y}, \tilde{z}$.
That is, we add divisors $\tilde{E}_k$ to $\tilde{G}$ for all $1 \leq k \leq 2d-2$ parametrizing formal curves $\tilde{y} = \tilde{u}_k \tilde{z}^k$.
Here $\tilde{u}_k$ is the preferred affine coordinate on $\tilde{E}_k$. There is an induced generically $2$-to-$1$ regular map
$$\delta_1 : \tilde{P}^{(p)} \to P^{(p)}$$
sending $\tilde{E}_k$ to $E_{k/2}$ for all even $k$, and contracting $\tilde{E}_k$ to $E_{(k-1)/2} \cap E_{(k+1)/2}$ for all odd $k$. Further, there is a map
$$\delta_2 : \tilde{P}^{(p)} \to P_+^{(p)}$$
that agrees with $\delta_1$ except that $\delta_2$ maps $\tilde{E}_{d-1}$ onto $E_{\Mid}$. There is an affine coordinate $u$ on $E_{\Mid}$ such that $\delta_2(\tilde{u}_{d-1}) = \tilde{u}^2_{d-1}.$

Let
$$s_2 = \delta^{\adj} \circ s \circ \delta,$$
$$r_2 = \delta^{\adj} \circ r \circ \delta,$$
$$\tilde{s} = \tau^{-1} \circ s_2 \circ \tau,$$
$$\tilde{r} = \tau^{-1} \circ r_2 \circ \tau.$$
Then $\delta_2$ semiconjugates $\tilde{s}$ to $\hat{s}_+$, and $\tilde{r}$ to $\hat{r}_+$. 

With this model, the proof of algebraic stability is completely analogous to the odd case. We give a (somewhat abbreviated) sketch showing how the statements translate.

\begin{lemma} \label{lem_r_mid}
    On the midpoint divisor, we have
    \begin{equation} \label{eq_rplus}
        (\hat{r}_+)_{\str} E_{\Mid} = E_{\Mid}
    \end{equation}
    and
\begin{equation} \label{eq_micro_rplus}
    \rho_+(u) = \frac{1}{u}.
\end{equation}
\end{lemma}

\begin{proof}
    Local equations for $r_2$ are obtained by substituting $y = \tilde{y}, z=\tilde{z}^2$ in Proposition \ref{prop_series} (\ref{item_prop_series_r}). We get
    $$\tilde{z}^2 = \tilde{z}'^2,$$
    $$\tilde{z}^{2d-2} (1 + O(\tilde{z}^{2d})) = \tilde{y} \tilde{y}'.$$
    Changing coordinates as in Lemma \ref{lem_r_A}, we get a multiplicity-$2$ equation for $\tilde{r}|_{\tilde{E}_{d-1} \vdash \tilde{E}_{d-1}}$,
    $$ \tilde{u}_{d-1} \tilde{u}_{d-1}' = 1.$$
    Semiconjugating by $\delta_2$, the equation for $\hat{r}_+|_{E_{\Mid} \vdash E_{\Mid}}$ is $u u' = 1$, proving \eqref{eq_rplus} and the formula for $\rho_+$.
\end{proof}

\begin{prop}
    The map $\hat{r}_+: P_+ \dashrightarrow P_+$ is biregular. 
\end{prop}

\begin{proof}
    The results of Lemmas \ref{lem_r_at_infty} and \ref{lem_r_A} except the claim about $\rho$ hold verbatim for $\hat{r}_+$ because all the computations are in local coordinates and can be done generically on the divisors. So Lemma \ref{lem_r_mid} completes the argument.
\end{proof}

Lemmas \ref{lem_s_infty} and \ref{lem_s_lower} hold verbatim for $\hat{s}_+$, again because these computations are local and computable generically on the divisors. Let us prove the analogue of Lemma \ref{lem_micro_s}.

\begin{lemma}\label{lem_micro_splus}
Let $\sigma_+ \colonequals \hat{s}_+|_{E_{\Mid} \vdash E_{\Mid}}.$
    The correspondence $\sigma_+$ is $(d-1)$-to-$(d-1)$,
    given by the formula
    \begin{equation}     \label{eq_micro_splus} 
    \sigma_+ (u) = \left\{ \frac{u}{\alpha^{d-1}} : \alpha \in \bk, \;
    \alpha^{d-1} + \ldots + 1 + d u = 0 \right\},
    \end{equation}
    and $\sigma_+(\infty) = \{-1/d\}$.
\end{lemma}

\begin{proof}
    Local equations for $s_2$ are obtained by substitution into Proposition \ref{prop_series} (\ref{item_prop_series_s}):
    $$\tilde{y} = \tilde{y}',$$
    $$-d\tilde{y}^2 = \tilde{z}^{2d-2} + \tilde{z}^{2d-4}\tilde{z}'^{2} + \ldots + \tilde{z}'^{2d-2} + A(\tilde{z}^{2d}, \tilde{z}'^{2d}).$$
    Calculating the strict transform as in the proof of \rev{Lemma \ref{lem_s_lower}} gives local equations for $\tilde{s}|_{\tilde{E}_{d-1} \vdash \tilde{E}_{d-1}}$:
    $$\tilde{u} = \tilde{u}' \tilde{\alpha}^{d-1},$$
    $$ -d \tilde{u}^2 = \tilde{\alpha}^{2d-2} + \tilde{\alpha}^{2d-4} + \ldots + 1,$$
    so again following our previous argument, we get $\tilde{s}|_{\tilde{E}_{d-1} \vdash \tilde{E}_{d-1}} (\infty) = \{\pm 1/\sqrt{-d} \}$.
    Then, by semiconjugacy, we obtain that $\sigma_+$ is $(d-1)$-to-$(d-1)$ and $\sigma_+(\infty) = \{-1/d\}$. 
    Finally, \eqref{eq_micro_splus} follows by semiconjugacy from the equations for $\tilde{s}|_{\tilde{E}_{d-1} \vdash \tilde{E}_{d-1}}$.
\end{proof}

\begin{lemma} \label{lem_ind_splus}
    The $\hat{s}_+$-contracted curves and $\hat{s}_+$-exceptional image curves are the divisors $E^{(p)}_k$, $1 \leq k < (d-1)/2$. For each $p \in \Ind r$,
    $$\Ind \hat{s}_+ \cap \mc{E}_+^{(p)} \subset E^{(p)}_{\Mid},$$
    $$ \Ind \hat{s}_+ \cap \mc{E}_+^{(p)} = \sigma_+(\infty) = \{-1/d\}.$$
    For each $u_0 \in \Ind \hat{s}_+ \cap \mc{E}_+^{(p)}$,
    we have
    $$ \hat{s}_+(u_0) = \bigcup_{1 \leq k \leq (d-1)/2} E^{(p)}_k.$$
\end{lemma}

\begin{proof}
    Follows from Lemma \ref{lem_micro_splus} using the proof of Lemma \ref{lem_ind_s}.
\end{proof}

\begin{prop} \label{prob_b_exc_ind_even}
    If $d$ is even,
    \rev{
    \[\Ind \hat{b}_+ \subset \bigcup_{p \in \Ind r} E^{(p)}_{\Mid},\]
    \[\Exc \hat{b}_+ \subset \bigcup_{p \in \Ind r} E^{(p)}_{\Mid}.\]
    }
    Given $p \in \Ind r$, under the identification $u: E^{(p)}_{\Mid}\to \PP^1$, 
    \rev{
    \[\Ind\hat{b}_+ \cap \mc{E}^{(p)} = \{-1/d \},\]
    \[\Exc\hat{b}_+ \cap \mc{E}^{(p)} = \{-d\}.\]
    }
    Further, there is a $(d-1)$-to-$(d-1)$ dominant restriction
    \[\beta_+ \colonequals \hat{b}_+|_{E_{\Mid}\vdash E_{\Mid}}.
    \] 
\end{prop}

One can check from Proposition \ref{prob_b_exc_ind_even} that $(P, \hat{b})$, obtained from $(P_+, \hat{b}_+)$ by contracting $E_{\Mid}^{(p)}$, is not algebraically stable in the even case, since $\Exc \hat{b} = \Ind \hat{b}$. This is our reason for working with $P_+$. 

\begin{lemma} \label{lem_invt_e}
Let $\bk = \C$, assume $d$ is even, and say $p \in \Ind r$. Let 
$$U = \C \smallsetminus \D = \{u: 1 \leq \abs{u} < \infty \} \subset E_{\Mid}^{(p)}.$$ Then $\hat{b}_+(U) \subset U$.
\end{lemma}

\begin{proof}
    By Proposition \ref{prob_b_exc_ind_even}, the theorem reduces to showing that $\hat{\beta}_+(U) \subset U$.
 By \eqref{eq_micro_rplus} and \eqref{eq_micro_splus},
we have
$$    \beta_+ (u) = \left\{ \frac{\alpha^{d-1}}{u} : \alpha \in \C, \;
    \alpha^{d-1} + \ldots + 1 + d u = 0 \right\}.$$
If $\abs{u} \geq 1$, $\abs{u'} < 1$, and $u' \in \beta_+(u)$, then there exists $\alpha \in \C$ satisfying
$$\abs{u'} = \abs{\frac{\alpha^{d-1}}{u} } < 1,$$
$$ \frac{\alpha^{d-1}}{u} + \ldots + \frac{1}{u} = -d.$$
The triangle inequality shows this is impossible, so we are done.
\end{proof}

We finally finish the proof of Theorem \ref{thm_main_model}. It has the same structure as the odd case.
\begin{thm} \label{thm_AS_even}
    Let $\bk = \C$ and $d \geq 2$ be even. Then $\hat{b}_+: P_+\vdash P_+$ is algebraically stable.
\end{thm}

\begin{proof}
By Theorem \ref{thm_dd_props} (\ref{it_as}), the correspondence $\hat{b}_+:P\vdash P$ is algebraically stable if for all $n\geq 0$,
\[\hat{b}_+^n(\Exc \hat{b}_+) \cap \Ind \hat{b}_+ = \varnothing.\]
Let $e_0 \in \Exc \hat{b}_+$. By Proposition \ref{prob_b_exc_ind_even}, $e_0 \in E_{\Mid}^{(p)} \cong \hat{\C}$ for some $p \in \Ind r$. Let $U = \C \smallsetminus \D \subset E_{\Mid}^{(p)}$. By Proposition \ref{prob_b_exc_ind_even}, $e_0 \in U$ and $\Ind \hat{b}_+ \cap U = \varnothing$. Since $\hat{b}_+(U) \subset U$ by Lemma \ref{lem_invt_e}, we have $\hat{b}_+^n(\Exc \hat{b}_+) \subset U$ for all $n \geq 0$, so $\hat{b}_+^n(\Exc \hat{b}_+) \cap \Ind \hat{b}_+ = \varnothing$.
\end{proof}

In principle, algebraic stability gives a second way of computing the dynamical degree $\lambda_1(b)$ for any given $d$, subject to computing $\hat{b}_*$ or $(\hat{b}_+)_*$. We prefer the method in Section \ref{sect_dd} because it is algebraic, avoids computing most of the matrix entries, and works for both even and odd $d$.

\subsection{The Ivrii Conjecture for generic algebraic plane regions}

We can now prove the Ivrii Conjecture for the Fermat hyperbola billiard and hence for the generic algebraic curve.

Given $n \in \N$, the algebraic billiard $b_{C,D}$ is \emph{$n$-reflective} if every $(x, v) \in C \times D$ satisfies
$$(x,v) \in b^n(x,v).$$

This is a more restrictive notion than found in the papers of Glutsyuk; in those papers, one uses $b(b(\ldots(b(x,v))\ldots))$, i.e. the $n$-th total composite of $b$ with itself.

\begin{thm}[= Theorem \ref{thm_main_ivrii}]
\leavevmode
\begin{enumerate}
    \item 
    Let $\bk = \C$ and $d \geq 2$. For all $n \in \N$, the Fermat hyperbola billiard $b_{C,D} : C \times D \vdash C \times D$ of degree $d$ is not $n$-reflective.
    \item
    Let $\bk = \C$ and $d \geq 2$. For all $n \in \N$, the generic algebraic billiard $b_{\gen} : C_{\gen} \times D \vdash C_{\gen} \times D$ is not $n$-reflective.
    \item
    Let $T \subset \R^2$ be a real algebraic plane curve defined by the vanishing of a polynomial $F(x,y)$ of degree $d \geq 2$ with algebraically independent coefficients over $\Q$. Let $\Omega$ be a bounded component of $\R^2 \smallsetminus T$. Then the set of periodic points of the classical billiard map inside $\Omega$ has measure $0$. 
\end{enumerate}

\end{thm}

\begin{proof} \leavevmode
    \begin{enumerate}
        \item 
     Assume by way of contradiction that $b_{C,D}$ is $n$-reflective. Then $\Gamma_{b^n}$ contains the diagonal. Say $d$ is odd. The graph $\Gamma_{\hat{b}^n}$ contains the diagonal by taking Zariski closures, so every point in $u_0 \in P$  satisfies $u_0 \in \hat{b}^n(u_0)$. Let $p \in \Ind r$, and let $U$ the invariant subset from Lemma \ref{lem_invt_o}; then the point $u_0 = \infty$ on $E_{\Mid}^{(p)}$ has the property $\hat{b}(u_0) \subset U$, and by Lemma \ref{lem_invt_o}, we have by induction that $\hat{b}^n(u_0) \subset U$. Since $u_0$ is not itself in $U$, we have a contradiction. If $d$ is even, the same argument applies using $(P_+, \hat{b}_+)$ and Lemma \ref{lem_invt_e}.
     \label{it_ivrii_pf_fh}
     
    \item This follows from (\ref{it_ivrii_pf_fh}), since if the generic algebraic billiard were $n$-reflective, then every specialization would be $n$-reflective. 

    \item Fix $n$. The space of algebraic curves of degree $d$ over $\bar{\Q}$ is naturally the scheme $\PP^\nu_{\bar{\Q}}$, where $\nu+1$ is the number of monomial terms appearing in a degree $d$ homogeneous equation in $3$ variables. 
    Let $\mc{B} \subset \PP^\nu_{\bar{\Q}}$ be the parameter space of smooth curves, and let $\mc{C} \to \mc{B}$ be the family of smooth degree $d$ plane curves. 
    The algebraic billiards correspondence extends to a dominant rational correspondence $\mc{C} \times D \vdash \mc{C} \times D$ over $\mc{B}$. Since the condition of being $n$-reflective is algebraic, the subscheme $\mc{B}_n$ of $\mc{B}$ consisting of $n$-reflective parameters is Zariski closed. We showed in the previous item that $\mc{B}_n$ does not contain the generic point of $\mc{B}$, so $\mc{B}_n$ is a proper subvariety; so the coordinates of any parameter in $\mc{B}_n(\C)$ are algebraically dependent over $\bar{\Q}$. It follows that $T_\C$ is not in $\mc{B}_n(\C)$, so the locus
    $$\mc{R}_n \colonequals \{(x,v) \in T_{\C} \times D : (x,v) \in b^n(x,v)\}$$
    is a proper subvariety of $T_\C \times D$. Hence $\mc{R}_n$ has dimension at most $1$ over $\C$.  

    By the assumption on algebraic independence of coefficients, the curve $T$ is smooth, so the classical billiard map $b_\Omega: W \to W$ inside $\Omega$ is defined. Here, $W$ is the set of inward-facing unit vectors based on $\partial \Omega$. Then
    $$\{(x,v) \in W : (x,v) = b_\Omega^n(x,v)\} \subset \mc{R}_n(\R).$$
    It follows that the set of $n$-periodic points of $b_\Omega$ is a union of real curves and points, and thus has measure $0$ in $W$. Taking the union over all of the countably many values of $n$ yields the result.
    \end{enumerate}
\end{proof}

The Ivrii conjecture is known for Whitney $\mc{C}^\infty$-generic tables, but it not obvious how to determine whether any particular table lies in the Whitney-generic set. On the other hand, we can manufacture explicit examples of curves to which Theorem \ref{thm_main_ivrii} (\ref{it_ivrii}) applies using the Lindemann-Weierstrass theorem. For instance, the curve
\begin{align*}
T : \quad 0=&e^{\sqrt{2}}x^{4}+0.3e^{\sqrt{3}}x^{3}y\ +\ e^{\sqrt{5}}x^{2}y^{2}+e^{\sqrt{6}}xy^{3}+0.3e^{\sqrt{7}}y^{4}+e^{\sqrt{10}}x^{3}+e^{\sqrt{11}}x^{2}y\ +\ \\
&e^{\sqrt{13}}xy^{2}+e^{\sqrt{14}}y^{3}+e^{\sqrt{15}}x^{2}+e^{\sqrt{17}}xy+e^{\sqrt{19}}y^{2}+e^{\sqrt{21}}x+0.3e^{\sqrt{22}}y+0.3e^{\sqrt{23}}
\end{align*}
has a bounded, non-convex component and is thus a new example of a curve satisfying the Ivrii Conjecture. The condition on algebraic independence of coefficients can be weakened -- we just need $T$ to map to the scheme-theoretic generic point of an algebraic family of curves containing the Fermat hyperbola. We think that the complex Ivrii conjecture could proved for other specific tables, which would establish the real Ivrii conjecture in more families.

\bibliographystyle{alpha}
\bibliography{bib}
\end{document}